\documentclass[11pt]{article}

\usepackage{amsmath, amssymb, amsthm, verbatim, enumerate, bbm, color} 
\usepackage{indentfirst}
\usepackage[vmargin=1in,hmargin=1in]{geometry}
\usepackage{enumitem}
\usepackage{graphicx}
\usepackage{tikz}

\date{\today}
\allowdisplaybreaks
\expandafter\let\expandafter\oldproof\csname\string\proof\endcsname
\let\oldendproof\endproof
\renewenvironment{proof}[1][\proofname]{%
	\oldproof[\bf #1]%
}{\oldendproof}

\newcommand{\ignore}[1]{}

\theoremstyle{plain}
\newtheorem{theorem}{Theorem}
\newtheorem{lemma}{Lemma}[section]
\newtheorem{claim}[lemma]{Claim}
\newtheorem{proposition}[lemma]{Proposition}

\newtheorem{corollary}[lemma]{Corollary}

\newtheorem{fact}[lemma]{Fact}

\DeclareMathOperator{\poly}{poly}

\newcommand{\Q}{\mathcal Q}

\newcommand{\ex}{\mathrm{ex}}
\newcommand{\maxcut}{\mathrm{maxcut}}

\RequirePackage[normalem]{ulem} 
\RequirePackage{color}\definecolor{RED}{rgb}{1,0,0}\definecolor{BLUE}{rgb}{0,0,1} 

\title{A Generalized Tur\'an Problem and its Applications}

\author{Lior Gishboliner \thanks{School of Mathematics, Tel Aviv University, Tel Aviv 69978, Israel. Email: liorgis1@post.tau.ac.il}
\and Asaf Shapira \thanks{
School of Mathematics, Tel Aviv University, Tel Aviv 69978, Israel.
Email: asafico$@$tau.ac.il. Supported in part by ISF Grant 1028/16 and ERC Starting Grant 633509.}
}

\date{}

\begin{document}

\maketitle

%

\begin{abstract}

The investigation of conditions guaranteeing the appearance of cycles of certain lengths is one of
the most well-studied topics in graph theory. In this paper we consider a problem of this type which
asks, for fixed integers ${\ell}$ and $k$,
how many copies of the $k$-cycle guarantee the appearance of an $\ell$-cycle?
Extending previous results of Bollob\'{a}s--Gy\H{o}ri--Li and Alon--Shikhelman, we fully resolve this problem by giving tight (or nearly tight) bounds for all values of $\ell$ and $k$.

We also present a somewhat surprising application of the above mentioned estimates to the study of the graph removal lemma.
Prior to this work, all bounds for removal lemmas were either polynomial or there was a tower-type gap between the best
known upper and lower bounds. We fill this gap by showing that for every super-polynomial function $f(\varepsilon)$,
there is a family of graphs ${\cal F}$, such that the bounds for the ${\cal F}$ removal lemma are precisely given by $f(\varepsilon)$.
We thus obtain the first examples of removal lemmas with tight super-polynomial bounds.
A special case of this result resolves a problem of Alon and the second author, while another special case partially resolves a problem of Goldreich.

\end{abstract}

\section{Introduction}\label{sec:intro}

\subsection{Background and the Main Result}

Tur\'an's Theorem \cite{Turan}, one of the cornerstone results in graph theory, determines the maximum number of edges in an $n$-vertex graph that does not contain a $K_t$ (the
complete graph on $t$ vertices). Tur\'an's problem is the following more general question: for a fixed graph
$H$ and an integer $n$, what is the maximum number of edges in an $n$-vertex $H$-free graph? This quantity is denoted by $\mbox{ex}(n,H)$.
Estimating $\mbox{ex}(n,H)$ for various graphs $H$ is one of the most well-studied problems in graph theory. See \cite{Simonovits_Survey} for a survey.

Alon and Shikhelman \cite{Alon_Shikhelman} have recently initiated the systematic study of the following natural generalization of $\mbox{ex}(n,H)$;
for fixed graphs $H$ and $T$, estimate $\mbox{ex}(n,T,H)$, which is the maximum number of copies\footnote{When
counting copies of $T$ in $G$ we always mean unlabeled copies.} of $T$ in an
$n$-vertex graph that contains no copy of $H$. Note that $\mbox{ex}(n,H)=\mbox{ex}(n,K_2,H)$.
While some special cases of this abstract problem have already been considered by Erd\H{o}s in the 60's,
\cite{Erdos62, Erdos84} contains
bounds on $\mbox{ex}(n,T,H)$ for various pairs $T,H$. For the sake
of brevity we refer the reader to \cite{Alon_Shikhelman} for more background and motivation, as well as examples
of some well-studied problems in extremal combinatorics which can be cast in the setting of studying $\mbox{ex}(n,T,H)$
for various pairs $H$ and $T$.

One special case of Tur\'an's problem which has been extensively studied is the estimation of
$\mbox{ex}(n,C_k)$ (where $C_k$ denotes the $k$-cycle, i.e., the cycle of length $k$).
While for odd $k$ it is known \cite{Simonovits} that
$\mbox{ex}(n,C_k)= \lfloor n^2/4 \rfloor$ (for large enough $n$), the problem of determining the order of magnitude of $\mbox{ex}(n,C_k)$
for even $k$ is a long standing open problem in extremal graph theory, see the survey \cite{Verstraete} and its references.

The problem of estimating $\ex(n,C_k,C_{\ell})$ has recently received a lot of attention.
Bollob\'{a}s and Gy\H{o}ri \cite{BG} proved that
$\ex(n,C_3,C_5) = \Theta\big( n^{3/2} \big)$. Gy\H{o}ri and Li \cite{GL} extended this result by considering $\ex(n,C_3,C_{2\ell+1})$. The dependence of their upper bound on $\ell$ was subsequently improved upon by Alon and Shikhelman \cite{Alon_Shikhelman}. At this moment, the best known bounds imply that
\begin{equation}\label{eq:C3}
\Omega( \ex(n,\{C_4,C_6,\dots,C_{2\ell}\})) \leq
\ex(n,C_3,C_{2\ell+1}) \leq
O(\ell \cdot \ex(n,C_{2\ell})),
\end{equation}
where $\ex(n,\{C_4,C_6,\dots,C_{2\ell}\})$ is the maximal number of edges in an $n$-vertex graph with no copy of $C_{2t}$ for any $2 \leq t \leq \ell$.
The lower bound above was proved in \cite{GL}, and the upper bound in \cite{Alon_Shikhelman}.
According to \cite{Alon_Shikhelman}, a result similar to \eqref{eq:C3} was also obtained by F\"{u}redi and \"{O}zkahya \cite{FO}.
The lower and upper bounds in \eqref{eq:C3} are known to be of the same order of magnitude,
$\Theta \big( n^{1+1/\ell} \big)$, for $\ell \in \{2,3,5\}$ (see e.g. \cite{Benson,Wenger}).
Another notable recent result is the {\em exact} determination of $\ex(n,C_5,C_3)$ by \cite{G,HHKNR}.

Our main result in this paper, stated as Theorem \ref{thm:main}, significantly extends the above results
by giving asymptotically tight bounds for $\ex(n,C_k,C_{\ell})$ for all fixed $k,\ell$.
In this theorem, as well as later on in the paper, we write $O_{k}/\Omega_{k}/\Theta_{k}$ to indicate that the notation hides constants which depend on $k$. When we write $O/\Omega/\Theta$, we mean that the implicit constants are absolute.

\begin{theorem}\label{thm:main}
For distinct $k,\ell$ we have
$$\ex(n,C_k,C_{\ell}) =
\begin{cases}
\Theta_{k}\big( n^{k/2} \big) &
k \geq 5, \ell = 4, \\
\Theta_{k}\big( \ell^{\lceil k/2 \rceil} n^{\lfloor k/2 \rfloor} \big)
& \ell \geq 6 \text{ even, } k \geq 4, \\
\Theta_{k}\big( \ell^{\lceil k/2 \rceil} n^{\lfloor k/2 \rfloor} \big)
& k,\ell \text{ odd, } 5 \leq k < \ell.
\end{cases}
$$
\end{theorem}

Observe that in the above theorem, our bounds are tight also when only $k$ is fixed.
This will be important in some of our applications.

Let us see which cases are not covered by Theorem \ref{thm:main} or by \eqref{eq:C3}.
Observe that if $k$ is even and $\ell$ is odd, or if $k$ and $\ell$ are both odd and $k > \ell$, then a blow-up of $C_k$ does not contain copies of $C_{\ell}$. Therefore, in these
cases\footnote{In these so-called ``dense" cases, while it is trivial to determine the order of magnitude of $\ex(n,C_k,C_{\ell})$, it may be very challenging to get an exact result, as in the case of $\ex(n,C_5,C_3)$ handled in \cite{G,HHKNR}.} we have $\ex(n,C_k,C_{\ell}) = \Theta_k(n^k)$. Thus, the only case remaining is $\ex(n,C_3,C_{2\ell})$, for which we will prove the following.
\begin{proposition}\label{prop:C3_even_cycle}
For every $\ell \geq 2$ we have
$$\Omega \big( \ex(n,\{C_4,C_6,\dots,C_{2\ell}\}) \big) \leq \ex(n,C_3,C_{2\ell}) \leq
O_{\ell}\big( \ex(n,C_{2\ell}) \big).$$
\end{proposition}
As in the case of \eqref{eq:C3}, the lower and upper bounds in Proposition \ref{prop:C3_even_cycle} are known to be of the same order of magnitude for $\ell \in \{2,3,5\}$.


\subsection{Exact Bounds for Removal Lemmas}\label{subsec:easy}

We would now like to describe a somewhat surprising application of Theorem \ref{thm:main}
related to the famous {\em graph removal lemma}. Let ${\cal P}$ be a monotone graph property, that is,
a property closed under removal of vertices and edges. We say that a graph $G$ is $\varepsilon$-far from
${\cal P}$ if one must remove at least $\varepsilon n^2$ edges to make $G$ satisfy ${\cal P}$.
Recall that the graph removal lemma \cite{RuzsaSz76} states that if $G$ is $\varepsilon$-far from being $H$-free then
$G$ contains at least $\delta_H(\varepsilon)n^h$ copies of $H$, where $h=|V(H)|$ and $\delta_{H}(\varepsilon)>0$.
Observe that this is equivalent to the statement that if $G$ is $\varepsilon$-far from being $H$-free then
a randomly chosen set of vertices $S$ of size $w_H(\varepsilon)$ is such that $G[S]$ (the subgraph of $G$ induced by $S$)
contains a copy of $H$ with probability at least $2/3$.

The latter way of stating the removal lemma naturally leads to the following generalization
that applies even to monotone properties that cannot be characterized by a finite number of forbidden subgraphs:
given a monotone graph property ${\cal P}$ and $\varepsilon >0$,
let $w_{\cal P}(\varepsilon)$ be the smallest integer (if one exists) so that if $G$ is $\varepsilon$-far from satisfying
${\cal P}$, then a random subset $S \subseteq V(G)$ of $w_{\cal P}(\varepsilon)$ vertices is such
that $G[S]$ does not satisfy ${\cal P}$ with probability at least $2/3$. In other words, $w_{\cal P}(\varepsilon)$ tells us how
many vertices we should sample from a graph that is $\varepsilon$-far from ${\cal P}$ in order to find a {\em witness}
(hence the notation $w_{\cal P}(\varepsilon)$) to this fact. For example, a result of R\"odl and Duke \cite{RD},
which resolved a conjecture of Erd\H{o}s \cite{Erdos}, states that
$w_{\cal P}(\varepsilon)$ is well-defined
when ${\cal P}$ is the $k$-colorability property. A significant generalization of the above results was obtained in \cite{AS_monotone},
where is was shown that $w_{\cal P}(\varepsilon)$ is well-defined for every monotone graph property ${\cal P}$.

It is known that there are many properties for which $w_{\cal P}(\varepsilon)=\mbox{poly}(1/\varepsilon)$, see, e.g., \cite{GS}.
However, there are many cases for which the best known bounds for $w_{\cal P}(\varepsilon)$ are anything but tight.
The prime example is the famous triangle removal lemma, which (as we mentioned above) is equivalent to the problem
of estimating $w_{\cal P}(\varepsilon)$ where ${\cal P}$ is the property of being triangle-free. Unfortunately, despite much
effort, the best known bounds for the triangle removal lemma are
\begin{equation}\label{eqremoval}
\left(1/\varepsilon\right)^{c\log 1/\varepsilon} \leq w_{\cal P}(\varepsilon) \leq \mbox{tower}(O(\log 1/\varepsilon))\;,
\end{equation}
where\footnote{$\mbox{tower}(x)$ is the tower function, that is, a tower of exponents of height $x$.}  the lower bound was already obtained in \cite{RuzsaSz76}, and the upper bound was first obtained only a few years ago by Fox \cite{Fox}
(see also \cite{ConlonFo13,MoshkovitzSh16} for alternative proofs). Observe that there is a tower-type gap between the lower and upper bounds in (\ref{eqremoval}). In fact, prior to this work, all known bounds for removal lemmas
were of the above two types, that is, they were either of the form $w_{\cal P}(\varepsilon)=\mbox{poly}(1/\varepsilon)$ or there was (at least)
a tower-type gap between the best known lower and upper bounds. For example, it was not known if there is a property satisfying (say) $w_{\cal P}(\varepsilon)=2^{\Theta(1/\varepsilon)}$ or one satisfying $w_{\cal P}(\varepsilon)=\mbox{tower}(\Theta(1/\varepsilon))$.
Our first application of Theorem \ref{thm:main} fills this gap by showing the following:

\begin{theorem}\label{thm:hierarchy}
There is an absolute constant $c$ such that for every decreasing function
$f\colon (0,1) \rightarrow \mathbb{N}$ satisfying $f(x) \geq 1/x$, there is a monotone graph property ${\cal P}$ satisfying
$
f(\varepsilon) \leq
w_{\cal P}(\varepsilon) \leq
\varepsilon^{-14}f(\varepsilon/c)
$.
\end{theorem}

As immediate corollaries of the above theorem we see that there is a monotone property ${\cal P}$ so that the bounds
for the ${\cal P}$ removal lemma satisfy $w_{\cal P}(\varepsilon)=2^{\Theta(1/\varepsilon)}$ (or analogously one satisfying $w_{\cal P}(\varepsilon)=\mbox{tower}(\Theta(1/\varepsilon))$). We again stress that this is the first example
of a removal lemma with a tight super-polynomial bound.

We now turn to the second application of Theorem \ref{thm:main}.
In what follows, we use $w_k(\varepsilon)$ instead of $w_{\cal P}(\varepsilon)$, where ${\cal P}$ is the $k$-colorability property.
The result of \cite{RD} mentioned above which proved that $w_k(\varepsilon)$ is well-defined relied on the regularity lemma \cite{Szemeredi}, and thus supplied only tower-type bounds.
A significantly better bound was obtained by Goldreich, Goldwasser and Ron \cite{GGR} who proved that $w_k(\varepsilon)=\mbox{poly}(1/\varepsilon)$. In a recent breakthrough, Sohler \cite{Sohler} obtained the nearly tight bound
$w_k(\varepsilon)=\tilde{\Theta}(1/\varepsilon)$.

Given a set of integers $L$, let us say that a graph is $L$-free if it is $C_{\ell}$-free for every $\ell \in L$.
The result of \cite{GGR} stating that $w_2(\varepsilon)=\mbox{poly}(1/\varepsilon)$ is then equivalent to the
statement that if $L$ consists of {\em all} odd integers then\footnote{We slightly abuse the notation here by using $L$ in the notation $w_{L}$ also for the {\em property} of being $L$-free.} $w_{L}(\varepsilon)=\mbox{poly}(1/\varepsilon)$.
Another related result is due to Alon \cite{Alon_Subgraphs}
who proved that $w_{L}(\varepsilon)=\mbox{poly}(1/\varepsilon)$ if $L$ contains at least one even integer,
and that $w_{L}(\varepsilon)$ is super-polynomial whenever $L$ is a {\em finite} set of odd integers.
Alon and the second author \cite{AS_monotone} asked if one can extend the above results by characterizing the sets of integers
$L$ for which $w_{L}(\varepsilon)=\mbox{poly}(1/\varepsilon)$.
Our second application of Theorem \ref{thm:main} solves all the cases not handled by previous results.

\begin{theorem}\label{thm:cycle_sequence_charac}
Let $L=\{\ell_1,\ell_2,\ldots\}$ be an infinite increasing sequence of odd integers.
Then
\begin{equation*}\label{eq:charac_condition}
w_{L}(\varepsilon)=\mathrm{poly}(1/\varepsilon) ~~~~\mbox{if and only if}~~~~ \limsup_{j \longrightarrow \infty}{\frac{\log \ell_{j+1}}{\log \ell_j}} < \infty.
\end{equation*}
\end{theorem}

By the above theorem, as long as $\ell_j$ does not grow faster than $2^{2^{j}}$, one can get a polynomial bound
on $w_{L}(\varepsilon)$, while for any (significantly) faster-growing $\ell_j$ one cannot get such a bound.

Let us give another perspective on the above result.
As mentioned earlier, it was shown in \cite{AS_monotone} that
$w_{\cal P}(\varepsilon)$ is well-defined for every monotone property ${\cal P}$.
Unfortunately, the general bounds obtained in \cite{AS_monotone} for $w_{\cal P}(\varepsilon)$ are of tower-type (or even worse).
A natural problem, first raised already in \cite{AS_monotone} and later also by  Goldreich \cite{Goldreich1} and Alon and Fox \cite{AF} is
to characterize the properties for which $w_{\cal P}(\varepsilon)=\mbox{poly}(1/\varepsilon)$. Since this problem seemed (and still seems) to be out of reach,
\cite{AS_monotone} asked if one can at least solve a (very) special case of this problem by characterizing the sets $L$ for which
$w_{L}(\varepsilon)=\mathrm{poly}(1/\varepsilon)$. This problem is resolved by Theorem \ref{thm:cycle_sequence_charac}.

\subsection{Applications to Graph Property Testing}

We now describe two applications of Theorem \ref{thm:main} in the area of graph property testing (see \cite{Goldreich1}
for more background and references on the subject). In the setting we will consider here, which was first introduced in \cite{GGR}, one assumes that a
(randomized) algorithm can sample a set of vertices $S$
from an input graph $G$, and based on the graph induced by this set should be able to distinguish\footnote{It was shown in \cite{GT} that one can transform every $\varepsilon$-tester into an $\varepsilon$-tester that works in this manner.} with probability at least
$2/3$ between the case that $G$ satisfies a property ${\cal P}$ and the case that $G$ is $\varepsilon$-far from satisfying
${\cal P}$. Such an algorithm is called an $\varepsilon$-tester for ${\cal P}$.
We denote by $q=q_{\cal P}(\varepsilon,n)$ the smallest integer for which there is an $\varepsilon$-tester for ${\cal P}$
that works by randomly selecting a set of vertices $S$ of size $q$ from $n$-vertex graphs.
The surprising fact is that for many natural properties ${\cal P}$ there is a function $q_{\cal P}(\varepsilon)$ satisfying
$q_{\cal P}(\varepsilon,n) \leq q_{\cal P}(\varepsilon)$ for every $n$, that is, there is an $\varepsilon$-tester for ${\cal P}$ which inspects an induced subgraph
of size that depends only on $\varepsilon$ and not on $|V(G)|$. The function $q_{\cal P}(\varepsilon)$ is sometimes referred to as the $2$-sided error\footnote{That is, the tester is allowed to err in both direction, i.e. reject graphs satisfying ${\cal P}$ with small probability as well as accept graphs that are $\varepsilon$-far from ${\cal P}$ with small probability.} query complexity
of the optimal\footnote{In the sense that it samples the smallest set of vertices $S$.} $\varepsilon$-tester of ${\cal P}$.

Observe that for a monotone property we have $q_{\cal P}(\varepsilon) \leq w_{\cal P}(\varepsilon)$, since the definition
of $w_{\cal P}(\varepsilon)$ guarantees that an $\varepsilon$-tester for ${\cal P}$ can simply sample a set $S$ of\footnote{The result of \cite{AS_monotone} mentioned above implies that $w_{\cal P}(\varepsilon)$ is well-defined for every ${\cal P}$.} $w_{\cal P}(\varepsilon)$
vertices and accept the input if and only if the graph induced by $S$ satisfies ${\cal P}$.
Note that while $\varepsilon$-testers might have $2$-sided error, such an algorithm for a monotone property has $1$-sided error, that is, it accepts graphs satisfying ${\cal P}$ with
probability $1$ (and rejects those that are $\varepsilon$-far from $\mathcal{P}$ with probability at least $2/3$).
It is easy to see that a $1$-sided tester of a monotone property ${\cal P}$ cannot reject an input if $G[S]$ satisfies
${\cal P}$. Hence $w_{\cal P}(\varepsilon)$ actually equals the query complexity of the optimal $1$-sided tester for ${\cal P}$.

As discussed in \cite{Goldreich1}, numerous problems of the above type have been studied in the past. Somewhat surprisingly,
all results in this area supplied either polynomial  or tower-type bounds for $q_{\cal P}(\varepsilon)$.
Goldreich \cite{Goldreich,Goldreich1} thus asked to exhibit a graph property with query complexity $q_{\cal P}(\varepsilon)=2^{\Theta(1/\varepsilon)}$.
The discussion in the previous paragraph together with the remark following Theorem \ref{thm:hierarchy}, supply a partial answer to this question
by exhibiting a property whose $1$-sided error query complexity is $2^{\Theta(1/\varepsilon)}$.

\begin{theorem}\label{corogold}
There is a graph property ${\cal P}$ with $1$-sided error query complexity
$w_{\cal P}(\varepsilon)=2^{\Theta(1/\varepsilon)}$\;.
\end{theorem}

Theorem \ref{thm:hierarchy} of course shows that for every sufficiently fast-growing $f$, there is a graph property
with $1$-sided error query complexity $w_{\cal P}(\varepsilon)=f(\Theta(\varepsilon))$. This can be considered
a {\em hierarchy theorem} for the query complexity of $1$-sided error $\varepsilon$-testers, somewhat reminiscent
of the famous time/space hierarchy theorems in computational complexity theory.

We now turn to the last application of Theorem \ref{thm:main}.
It is natural to ask whether $q_{\cal P}(\varepsilon)$ can be significantly smaller that $w_{\cal P}(\varepsilon)$,
that is, if $2$-sided testers have any advantage over $1$-sided ones.
The simple answer is of course yes; for example, if ${\cal P}$ is the property of having edge density at least $1/4$ (i.e. having at least $n^2/4$ edges)
then it is easy to see
that $q_{\cal P}(\varepsilon) \leq \poly(1/\varepsilon)$ as one can just estimate the edge
density of the input. On the other hand, it is also easy to see that ${\cal P}$ is not testable with $1$-sided error using a number
of queries that is independent of $n$. It is thus more natural to
restrict ourselves to graph properties that {\em can} be tested with $1$-sided error, and ask:
to what extent are $2$-sided testers more powerful than $1$-sided testers for monotone properties?

It is (perhaps) natural to guess that at least for monotone
properties ${\cal P}$, $2$-sided testers do not have any advantage over $1$-sided testers, the explanation being
that the only way one can be convinced that a graph is far from satisfying a monotone property ${\cal P}$ is by finding a witness to
this fact in the form of a subgraph not satisfying ${\cal P}$. As Theorem \ref{thm:two_sided_test} below shows, this intuition turns out to be false in a very strong sense.
This theorem implies that $2$-sided testers can be {\em arbitrarily} more efficient than $1$-sided testers,
even for monotone graph properties. Prior to this work, it was not even known that $2$-sided testers can be super-polynomially
stronger than $1$-sided testers.

\begin{theorem}\label{thm:two_sided_test}
For every decreasing function $f \colon (0,1) \rightarrow \mathbb{N}$ satisfying
$f(x) \geq 1/x$, there is a monotone graph property ${\cal P}$ so that
\begin{itemize}
\item ${\cal P}$ has $1$-sided error query complexity $w_{\cal P}(\varepsilon) \geq f(\varepsilon)$.
\item ${\cal P}$ has $2$-sided error query complexity $q_{\cal P}(n,\varepsilon) = \poly(1/\varepsilon)$ for every $n \geq n_0(\varepsilon)$.
\end{itemize}
\end{theorem}

Let us give another perspective on the above theorem.
Observe that the fact that a property can be $\varepsilon$-tested using query complexity $q_{\cal P}(\varepsilon)$ is equivalent to the assertion
that the distribution of induced subgraphs on $q_{\cal P}(\varepsilon)$ vertices obtained by drawing
$q_{\cal P}(\varepsilon)$ vertices from a graph in ${\cal P}$ is distinguishable\footnote{This interpretation is reminiscent of the way one studies limits of dense and sparse graph sequences, see \cite{Lovasz}.} from the distribution obtained by drawing these vertices from a graph that is $\varepsilon$-far from ${\cal P}$. So the above theorem implies that there is a monotone graph property ${\cal P}$ and a graph $G$ that is $\varepsilon$-far from ${\cal P}$ so that even though {\em all} subsets of vertices of $G$ of size (say) $2^{1/\varepsilon}$ do satisfy ${\cal P}$, the distribution of induced subgraphs on $\mbox{poly}(1/\varepsilon)$ vertices drawn from $G$ is distinguishable from the one
drawn from a graph satisfying ${\cal P}$. In other words, we can detect the fact that $G$ does not satisfy ${\cal P}$ without actually finding
a proof of this fact.

\paragraph{Paper organization:}

In Section \ref{sec:consec_odds} we give a tight upper bound for $\ex(n,C_{2k+1},C_{2k+3})$ where $k \geq 2$, which turns out to require
a different argument than the one needed to handle all other cases of Theorem \ref{thm:main}. This problem appears to be significantly harder
than $\ex(n,C_{3},C_{5})$ which was resolved by Bollob\'{a}s--Gy\H{o}ri \cite{BG}. This is best evidenced by the fact
that while $\ex(n,C_{3},C_{5})=\Theta(n^{3/2})$, for the general problem we have $\ex(n,C_{2k+1},C_{2k+3})=\Theta_k(n^k)$ for $k \geq 2$.
Section \ref{sec:proof main} contains the proof of Theorem \ref{thm:main} and Proposition \ref{prop:C3_even_cycle}. In this section we also prove a tight bound $\mbox{ex}(n,P_k,C_{\ell})$ for all values of
$k \geq 2$ and $\ell$, where $P_k$ is the path with $k$ edges (see Theorem \ref{thm:ex_path_cycle}).
In Section \ref{sec:test} we apply our bounds for $\mbox{ex}(n,C_k,C_{\ell})$ in order to prove Theorems \ref{thm:hierarchy}, \ref{thm:cycle_sequence_charac} and \ref{thm:two_sided_test}.
Lemma \ref{lem:P3 even_cycle main},
which is the key lemma in the proof of Theorem \ref{thm:main}, is proved in Section \ref{sec:key}. The main tool used in its proof is a bound for the skew version of the even-cycle Tur\'an problem,
due to Naor and Verstra\"{e}te \cite{NV}. 
Finally, in Section \ref{sec:lower} we prove the lower bounds in Theorem \ref{thm:main} and Proposition \ref{prop:C3_even_cycle}.

The dependence of our bounds on $\ell$ is important due to the way we apply them in Section \ref{sec:test}. We made little effort, however, to optimize their dependence on $k$.
Finally, since in most arguments the parity of the cycle lengths will be important, we will use $2k$ or $2k+1$ (and analogously $2\ell$ or $2\ell+1$)
to denote the lengths of the cycles.

\section{The Case $\ex(n,C_{2k+1},C_{2k+3})$}\label{sec:consec_odds}
In this section we give a tight upper bound for $\ex(n,C_{2k+1},C_{2k+3})$ when $k \geq 2$.
Let us introduce some notation that we will use throughout the paper.
For a graph $G$ and disjoint sets $X,Y \subseteq V(G)$, we denote by $E(X,Y)$ the set of edges with one endpoint in $X$ and one endpoint in $Y$, and set
$e(X,Y) = |E(X,Y)|$. For $v \in V(G)$ and $X \subseteq V(G)$, denote $N_X(v) = \{x \in X : (v,x) \in E(G)\}$.

Let $U_1,\dots,U_s$ be disjoint vertex sets in a graph.
A {\em $(U_1,\dots,U_s)$-path} is a path $u_1,\dots,u_s$ with $u_i \in U_i$.
Similarly, a {\em $(U_1,\dots,U_s)$-cycle} is a cycle $u_1,\dots,u_s,u_1$ with $u_i \in U_i$.
Let $p(U_1,\dots,U_s)$ denote the number of $(U_1,\dots,U_s)$-paths and let $c(U_1,\dots,U_s)$ denote the number of $(U_1,\dots,U_s)$-cycles. We denote by $P_k$ the path of length $k$, where the length of a path is the number of edges in it. We will frequently use the following simple averaging argument.
\begin{claim}\label{prop:random_partition}
	Let $G$ be a graph. If for every partition $V(G) = U_1 \cup \dots \cup U_k$ it holds that $c(U_1,\dots,U_k) \leq r$, then the number copies of $C_k$ in $G$ is at most
	$\frac{1}{2}k^{k-1} r$. Similarly, if for every partition $V(G) = U_1 \cup \dots \cup U_k$ it holds that $p(U_1,\dots,U_k) \leq r$, then the number of copies of $P_{k-1}$ in $G$ is at most $\frac{1}{2}k^k r$.
\end{claim}
\begin{proof}
	Let $V(G) = U_1 \cup \dots \cup U_k$ be a random partition, generated according to $\mathbb{P}[v \in U_i] = \frac{1}{k}$ for each $v \in V(G)$ and
	$1 \leq i \leq k$, independently. Then
	$\mathbb{E}\left[ c(U_1,\dots,U_k) \right] = \#C_k(G) \cdot 2k \cdot k^{-k}$ and $\mathbb{E}\left[ p(U_1,\dots,U_k) \right] = \#P_{k-1}(G) \cdot 2 \cdot k^{-k}$, where $\#C_k(G)$ (resp. $\#P_{k-1}(G)$) denotes the number of copies of $C_k$ (resp. $P_{k-1}$) in $G$. Since these expectations are not larger than $r$, the claim follows.
\end{proof}
In what follows, let us denote the vertices of $C_{2k+1}$ (the $(2k+1)$-cycle) by $1,\dots,2k+1$, with edges $\{1,2\},\dots,\{2k,2k+1\},\{2k+1,1\}$.
For a graph $G$, denote by $\mathcal{I}(G)$ the set of all non-empty independent sets of $G$. We will need the following trivial (yet somewhat complicated to state) claim.
\begin{claim}\label{claim:maximal_independent_sets}
Let $J$ be a non-empty independent set of $C_{2k+1}$. Then there is $I \in \mathcal{I}(C_{2k+1})$ which contains $J$ and satisfies the following. Let $i_1,\dots,i_r$ be the elements of $I$ in the order they appear when traversing the cycle $1,\dots,2k+1$. Then for every $1 \leq j \leq r$, $i_j$ and $i_{j+1}$ are at distance either $2$ or $3$, namely $i_{j+1}-i_j \equiv 2,3 \pmod{2k+1}$, and if $i_j$ and $i_{j+1}$ are at distance $3$ then either $i_j \in J$ or $i_{j+1} \in J$.
\end{claim}
\begin{proof}
	If $|J| = 1$, say without loss of generality $J = \{1\}$,
	then
	$I = \{2j-1 : 1 \leq j \leq \noindent k\}$ is easily seen to satisfy the requirements of the claim. Assume then that $|J| \geq 2$, and
	let $j_1,\dots,j_r$ be the elements of $J$, as they appear when traversing the $(2k+1)$-cycle $1,\dots,2k+1$. For each $1 \leq i \leq r$, we greedily pick an independent set $I_i$ in the path connecting $j_i$ and $j_{i+1}$, which contains both $j_i$ and $j_{i+1}$, as follows.
	In addition to $j_i$ and $j_{i+1}$, we add to $I_i$ the elements $j_{i}+2,j_{i}+4,\dots$ until we reach $j_{i+1}$ or $j_{i+1} - 1$. If we reached $j_{i+1}$, then the distance between every pair consecutive elements of $I_i$ is $2$, and if we reached $j_{i+1} - 1$ then this true for all pairs except for $j_{i+1}-3,j_{i+1}$. It is now easy to see that $I = \bigcup_{i = 1}^{r}{I_i}$ satisfies the requirements of the claim.
\end{proof}
\begin{lemma}\label{lem:2k+1, 2k+3}
For every $k \geq 2$ it holds that
$\ex(n,C_{2k+1},C_{2k+3}) \leq (2k+1)^{2k}2^{2k+1}n^k$.
\end{lemma}
\begin{proof}
	Let $G$ be an $n$-vertex $C_{2k+3}$-free graph. By claim \ref{prop:random_partition} it is sufficient to prove that for every partition
	$V(G) = U_1 \cup \dots \cup U_{2k+1}$ we have
	$c(U_1,\dots,U_{2k+1}) \leq 2^{2k+1}n^k$. We will actually prove that
	\begin{equation}\label{eq:2k+1,2k+3_upper_bound}
	c(U_1,\dots,U_{2k+1}) \leq \sum_{I \in \mathcal{I}(C_{2k+1})}{\prod_{i \in I}{|U_i|}}.
	\end{equation}
	This will be sufficient, as $C_{2k+1}$ has at most $2^{2k+1}$ independent sets, and each of these sets contributes at most $n^k$ to the above sum.
	Assume by contradiction that \eqref{eq:2k+1,2k+3_upper_bound} is false.
	Let $\mathcal{C}$ denote the set of all $(U_1,\dots,U_{2k+1})$-cycles in $G$.
	We first show that there is $C = (u_1,\dots,u_{2k+1}) \in \mathcal{C}$ such that for every
	$I \in \mathcal{I}(C_{2k+1})$ there is
	$C' \in \mathcal{C} \setminus \{C\}$ which contains
	$\{u_i : i \in I\}$. We find $C$ greedily as follows. As long as there is $C = (u_1,\dots,u_{2k+1}) \in \mathcal{C}$ and
	$I \in \mathcal{I}(C_{2k+1})$ such that $C$ is the only $(U_1,\dots,U_{2k+1})$-cycle containing $\{u_i : i \in I\}$, we remove $C$ from $\mathcal{C}$, and we say that $C$ was removed due to
	$\{u_i : i \in I\}$. Fixing any
	$I \in \mathcal{I}(C_{2k+1})$
	and $u_i \in U_i$ for $i \in I$,
	observe that at most one cycle from $\mathcal{C}$ was removed due to
	$\{u_i : i \in I\}$. Thus, the overall number of cycles removed is not larger than
	the right-hand side of \eqref{eq:2k+1,2k+3_upper_bound}.
	Since by our assumption \eqref{eq:2k+1,2k+3_upper_bound} is false, there is a cycle
	$C = (u_1,\dots,u_{2k+1}) \in \mathcal{C}$ which had not been removed by the end of the process. Then $C$ satisfies our requirement. Let us fix such a $C = (u_1,\dots,u_{2k+1})$ for the rest of the proof.
	
	Let $J$ be the set of all
	$1 \leq i \leq 2k+1$ such that there is
	$u'_i \in U_i \setminus \{u_i\}$ which is adjacent to $u_{i-1}$ and $u_{i+1}$. We claim that $J$ is a non-empty independent set (of the $(2k+1)$-cycle).
	To show that $J$ is an independent set, assume by contradiction that there is $1 \leq i \leq 2k+1$ such that $i,i+1 \in J$, and let
	$u'_i \in U_i \setminus \{u_i\}$ and
	$u'_{i+1} \in U_{i+1} \setminus \{u_{i+1}\}$
	be witnesses to $i,i+1 \in J$. Then $u'_i,u_{i+1},u_i,u'_{i+1},u_{i+2},\dots,u_{i-1},u'_i$
	is a $(2k+3)$-cycle, a contradiction.
	We now show that $J \neq \emptyset$. Set
	$I' = \{2j : 2 \leq j \leq k\} \cup \{1\}$ and
	$I'' = \{2j : 3 \leq j \leq k\} \cup \{1,3\}$ and note that they are both independent sets.
	By our choice of $C = (u_1,\dots,u_{2k+1})$, there is
	$C' = (u'_1,\dots,u'_{2k+1}) \in \mathcal{C} \setminus \{C\}$ which contains $u_i$ for every
	$i \in I'$. Since $C' \neq C$, one of the following holds:
	either $u'_i \neq u_i$ for some $i \in \{2j+1 : 2 \leq j \leq k\}$, implying that $i \in J$ and we are done, or $(u'_2,u'_3) \neq (u_2,u_3)$.
	If $u'_2 = u_2$ or $u'_3 = u_3$ then $3 \in J$ or $2 \in J$, respectively, and again we are done. We deduce that $u'_2 \neq u_2$ and $u'_3 \neq u_3$.
	By repeating the same argument with respect to $I''$, we get a cycle
	$C'' = (u''_1,\dots,u''_{2k+1}) \in \mathcal{C} \setminus \{C\}$ such that either $u''_i \neq u_i$ for some
	$i \in \{2j+1 : 3 \leq j \leq k\} \cup \{2\}$, implying that
	$i \in J$ and we are done, or $u''_4 \neq u_4$ and
	$u''_5 \neq u_5$.
	But now
	$u_1,u'_2,u'_3,u_4,u_3,u''_4,u''_5,u_6,\dots,u_{2k+1},u_1$
	is a $(2k+3)$-cycle, a contradiction. See the top drawing in Figure 1 for an illustration.
	
	We thus proved that $J$ is a non-empty independent set. Apply Claim \ref{claim:maximal_independent_sets} to $J$ to get $I \in \mathcal{I}(C_{2k+1})$ with the properties stated in the claim.
	By our choice of $C = (u_1,\dots,u_{2k+1})$, there is
	$C' = (u'_1,\dots,u'_{2k+1}) \in \mathcal{C} \setminus \{C\}$ which contains $u_i$ for every
	$i \in I$. Let $i_1,\dots,i_r$ be the elements of $I$ in the order they appear when traversing the cycle $1,\dots,2k+1$. Since $C' \neq C$, there is
	$1 \leq j \leq r$ such that
	$(u'_{i_j + 1},\dots,u'_{i_{j+1}-1}) \neq (u_{i_j + 1},\dots,u_{i_{j+1}-1})$
	\nolinebreak \footnote{Here subscripts are taken modulo $2k+1$, while double subscripts are taken modulo $r$.}.
	Assume without loss of generality that $j = 1$ and $i_1 = 2$ (so in particular, $2 \in I$).
	By the guarantees of Claim \ref{claim:maximal_independent_sets}, we have
	$i_2 - i_1 \equiv 2,3 \pmod{2k+1}$, so either $i_2 = 4$ or $i_2 = 5$.
	Assume first that
	$i_2 = 4$. Then
	$u'_{3} \neq u_{3}$, implying that $3 \in J$, which is impossible as
	$2 \in I$, $J \subseteq I$ and $I$ is an independent set.
	Assume now that
	$i_2 = 5$. If $u'_{3} = u_{3}$ then $u'_4 \neq u_4$ and so
	$4 \in J$, which is again impossible as $5 \in I$, $J \subseteq I$ and $I$ is an independent set. So
	$u'_3 \neq u_3$ and similarly
	$u'_4 \neq u_4$. By the guarantees of Claim \ref{claim:maximal_independent_sets}, we have that either
	$2 \in J$ or $5 \in J$, say without loss of generality that $2 \in J$. Then by the definition of $J$, there is
	$u''_2 \in U_2 \setminus \{u_2\}$ adjacent to $u_{1}$ and
	$u_{3}$. But now
	$u_1,u''_2,u_3,u_2,u'_3,u'_4,u_5,\dots,u_{2k+1},u_1$ is a $(2k+3)$-cycle, a contradiction. See the bottom drawing in Figure 1 for an illustration.
	This completes the proof.	
\end{proof}
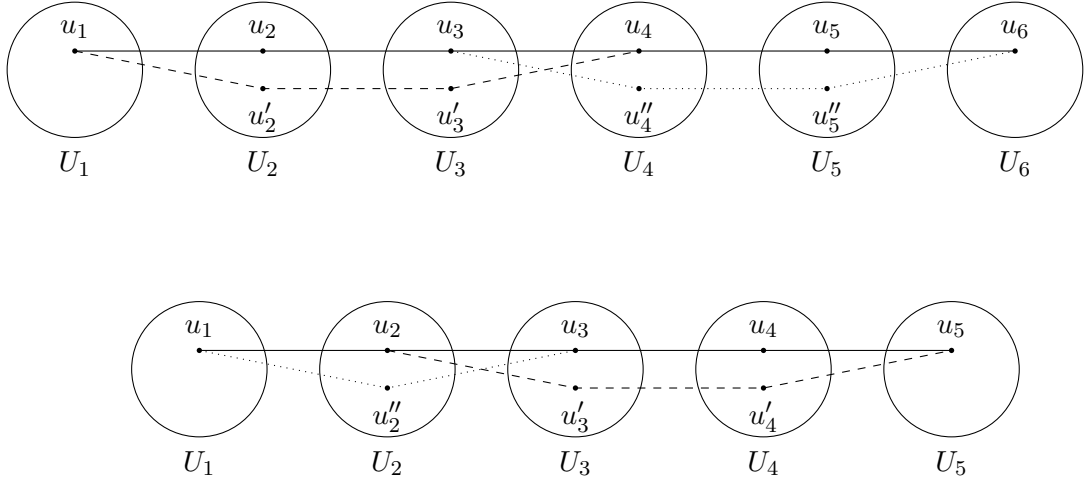
\begin{figure}[h]\label{figure:2k+1,2k+3}
	\centering
	\begin{tikzpicture}
	\foreach \i in {1,2,3,4,5,6}
	{
		\coordinate (a\i) at (2.5*\i,0.25);
		\coordinate (b\i) at (2.5*\i,-0.25);
		\coordinate (c\i) at (2.5*\i,0);
		\draw (a\i) node[fill=black,circle,minimum size=2pt,inner sep=0pt,label=$u_\i$] {};
		\draw (c\i) circle (0.9cm);
		\draw (c\i)+(0,-1.25) node {$U_\i$};
	}
	\foreach \i in {2,3}
	{
		\draw (b\i) node[fill=black,circle,minimum size=2pt,inner sep=0pt,label=below:$u'_\i$] {};
	}
	\foreach \i in {4,5}
	{
		\draw (b\i) node[fill=black,circle,minimum size=2pt,inner sep=0pt,label=below:$u''_\i$] {};
	}
	
	\draw (a1) -- (a2) -- (a3) -- (a4) -- (a5) -- (a6);
	\draw [dashed] (a1) -- (b2) -- (b3) -- (a4);
	\draw [dotted](a3) -- (b4) -- (b5) -- (a6);
	\end{tikzpicture}
	\newline \newline \newline \newline
	\begin{tikzpicture}
	\foreach \i in {1,2,3,4,5}
	{
		\coordinate (a\i) at (2.5*\i,0.25);
		\coordinate (b\i) at (2.5*\i,-0.25);
		\coordinate (c\i) at (2.5*\i,0);
		\draw (a\i) node[fill=black,circle,minimum size=2pt,inner sep=0pt,label=$u_\i$] {};
		\draw (c\i) circle (0.9cm);
		\draw (c\i)+(0,-1.25) node {$U_\i$};
	}
	\foreach \i in {2}
	{
		\draw (b\i) node[fill=black,circle,minimum size=2pt,inner sep=0pt,label=below:$u''_\i$] {};
	}
	\foreach \i in {3,4}
	{
		\draw (b\i) node[fill=black,circle,minimum size=2pt,inner sep=0pt,label=below:$u'_\i$] {};
	}
	
	\draw (a1) -- (a2) -- (a3) -- (a4) -- (a5);
	\draw [dotted] (a1) -- (b2) -- (a3);
	\draw [dashed] (a2) -- (b3) -- (b4) -- (a5);
	\end{tikzpicture}
	\caption{Illustrations for the proof of Lemma \ref{lem:2k+1, 2k+3}}
\end{figure}


\section{Proof of the Main Result}\label{sec:proof main}
In this section we prove the upper bounds for all cases in Theorem \ref{thm:main}, except for the case of two consecutive odd integers which was handled in Section \ref{sec:consec_odds}. The lower bounds will be proven in Section \ref{sec:lower}. At the end of this section, we give the proof of Proposition \ref{prop:C3_even_cycle}.

\subsection{Preliminary Lemmas}\label{subsec:prelim}

Here we introduce several lemmas which will be used in the proof of Theorem \ref{thm:main}. 
We start with the following key lemma, which is the most important ingredient in the proof of Theorem \ref{thm:main}, as it allows us to obtain tight bounds in terms of $n$ and $\ell$. The proof of this lemma appears in Section \ref{sec:key}.
\begin{lemma}\label{lem:P3 even_cycle main}
	Let $\ell \geq 3$, let $G$ be an $n$-vertex graph, let
	$X,Y,Z,W \subseteq V(G)$ be pairwise-disjoint vertex-sets and assume that the bipartite graphs $(X,Y)$, $(Y,Z)$ and $(Z,W)$ are $C_{2\ell}$-free.
	Then there are subsets
	$Y' \subseteq Y$ and $Z' \subseteq Z$ such that
	\begin{enumerate}
		\item $e(Y',X), e(Y',Z), e(Z',Y), e(Z',W) = O(\ell n)$.
		\item $p(X, Y \setminus Y', Z \setminus Z', W) = O(\ell^2 n^2)$.
	\end{enumerate}
\end{lemma}
At the end of Section \ref{sec:key}, we explain why the sets $Y'$ and $Z'$ in the statement of Lemma 
\ref{lem:P3 even_cycle main} are required, and why Lemma \ref{lem:P3 even_cycle main} is false for $\ell = 2$. The falsity of Lemma \ref{lem:P3 even_cycle main} for $\ell = 2$ is the reason we need a separate proof for the case $\ex(n,C_{2k+1},C_{2k+3})$ (see Section \ref{sec:consec_odds}).

In what follows we will need a special case of the following theorem, which gives a tight bound on $\ex(n,P_k,C_{2\ell})$ for every $k \geq 2$. The proof of this theorem appears at the end of this section. 
\begin{theorem}\label{thm:ex_path_cycle}
	For every $k \geq 2$, we have
	\begin{equation*}
	\ex(n,P_k,C_{2\ell}) =
	\begin{cases}
	\Theta_k(n^{k/2 + 1})
	& \ell = 2, \\
	\Theta_k( \ell^{\lfloor (k+1)/2 \rfloor} n^{\lceil (k+1)/2 \rceil})
	& \ell \geq 3.
	\end{cases}
	\end{equation*}
\end{theorem}
To complement Theorem \ref{thm:ex_path_cycle}, note that
$\ex(n,P_k,C_{2\ell+1}) = \Theta_k(n^{k+1})$, since a blowup of $P_k$ does not contain odd cycles.
The following lemma also plays a key role in the proof of Theorem \ref{thm:main}. 
\begin{lemma}\label{lem:paths_bound main}
	Let $s \geq 2$ and $\lambda \geq 1$, let $G$ be an $n$-vertex graph and let
	$U_1,\dots,U_s \subseteq V(G)$ be pairwise-disjoint sets such that
	$e(U_1,U_2) \leq \lambda(|U_1| + |U_2|)$ and
	$e(N_{U_{i+1}}(u_i),U_{i+2}) \leq
	\lambda(\left| N_{U_{i+1}}(u_i) \right| + \left| U_{i+2} \right|)$ for every
	$1 \leq i \leq s-2$ and $u_i \in U_i$.
	Then
	$$
	p(U_1,\dots,U_s) \leq
	\begin{cases}
	\lambda^{(s-1)/2}n^{(s-3)/2}\left( |U_1||U_s| + \lambda n \right)
	& s \text{ is odd},
	\\
	\lambda^{s/2}n^{s/2-1}(|U_1| + |U_2|)
	& s \text{ is even}.
	\end{cases}
	$$
\end{lemma}
\begin{proof}
	The proof is by induction on $s$. The base case $s=2$ is given by our assumption that 
	$e(U_1,U_2) \leq \lambda(|U_1| + |U_2|)$. Let then $s \geq 3$. Note that for every $u_1 \in U_1$, the sets $N_{U_2}(u_1),U_3,\dots,U_s$ satisfy the assumptions of the lemma, so we may apply the induction hypothesis to them. Suppose first that $s$ is odd. We have
	\begin{align*}
	p(U_1,\dots,U_s) &=
	\sum\limits_{u_1 \in U_1}{p(N_{U_2}(u_1),U_3,\dots,U_s)} \leq 
	\sum\limits_{u_1 \in U_1}{\lambda^{(s-1)/2}n^{(s-3)/2}(|N_{U_2}(u_1)| + |U_s|)} \\ &= 
	\lambda^{(s-1)/2}n^{(s-3)/2} \cdot ( e(U_1,U_2) + |U_1||U_s| ) \leq 
	\lambda^{(s-1)/2}n^{(s-3)/2} \cdot (\lambda(|U_1|+|U_2|) + |U_1||U_s|) \\ &\leq 
	\lambda^{(s-1)/2}n^{(s-3)/2} \cdot ( |U_1||U_s| + \lambda n ),
	\end{align*}
	where in the first inequality we used the induction hypothesis for $s-1$, and in the second inequality we used the assumption $e(U_1,U_2) \leq \lambda(|U_1| + |U_2|)$. The induction step for even $s$ is similar. Indeed, 
	\begin{align*}
	p(U_1,\dots,U_s) &=
	\sum\limits_{u_1 \in U_1}{p(N_{U_2}(u_1),U_3,\dots,U_s)} \leq 
	\sum\limits_{u_1 \in U_1}{\lambda^{(s-2)/2}n^{(s-4)/2}(|N_{U_2}(u_1)||U_s| + \lambda n)} \\ &= 
	\lambda^{(s-2)/2}n^{(s-4)/2} \cdot e(U_1,U_2) \cdot |U_s| + \lambda^{s/2}n^{s/2-1} \cdot |U_1| \\ &\leq 
	\lambda^{(s-2)/2}n^{(s-4)/2} \cdot \lambda(|U_1| + |U_2|) \cdot |U_s| + \lambda^{s/2}n^{s/2-1} \cdot |U_1| \leq  
	\lambda^{s/2}n^{s/2-1} \cdot (|U_1| + |U_s|),
	\end{align*}
	where in the first inequality we used the induction hypothesis for $s-1$, in the second inequality we used the assumption $e(U_1,U_2) \leq \lambda(|U_1| + |U_2|)$, and in the last inequality we used the trivial bound $|U_1|+|U_2| \leq n$. 
\end{proof}
We now derive two important corollaries of Lemma \ref{lem:paths_bound main}, stated as Lemmas 
\ref{lem:paths_bound forbidden_path} and \ref{lem:paths_bound forbidden_even_cycle} below. In their proof we will use the following well-known theorem of Erd\H{o}s and Gallai.
\begin{theorem}[\cite{EG}]\label{thm:EG}
	For every $t \geq 1$ we have $\ex(n,P_{t}) \leq \frac{t - 1}{2}n$.
\end{theorem}
\begin{lemma}\label{lem:paths_bound forbidden_path}
	Let $2 \leq s < t$ be integers having the same parity, let $G$ be an $n$-vertex graph and let
	$U_1,\dots,U_s \subseteq V(G)$ be pairwise-disjoint vertex-sets such that there is no path of length $t-1$ inside $U_1 \cup \dots \cup U_s$ between a vertex in $U_1$ and a vertex in $U_s$. Then
	$$
	p(U_1,\dots,U_s) \leq
	\begin{cases}
	\left(\frac{t-s}{2}\right)^{(s-1)/2}n^{(s-3)/2}
	\left( |U_1||U_s| + \frac{t-s}{2}n \right)
	& s \text{ is odd},
	\\
	\left(\frac{t-s}{2}\right)^{s/2}n^{s/2-1}
	\left( |U_1| + |U_2| \right)
	& s \text{ is even}.
	\end{cases}
	$$
\end{lemma}
\begin{proof}
	We may and will assume that every edge in $G$ is on some $(U_1,\dots,U_s)$-path (as deleting all other edges does not change $p(U_1,\dots,U_s)$). It is sufficient to show that the conditions of Lemma \ref{lem:paths_bound main} hold for $\lambda = \frac{t-s}{2} \geq 1$. We prove the stronger statement that for every $1 \leq i \leq s-1$ and for every
	$U'_i \subseteq U_i$ and $U'_{i+1} \subseteq U_{i+1}$, it holds that
	$e(U'_i,U'_{i+1}) \leq \frac{t-s}{2}\left( |U'_i| + |U'_{i+1}| \right)$.
	If, by contradiction, this does not hold, then by Theorem \ref{thm:EG} there is a path
	$P = v_1,\dots,v_{t-s+2}$ of length $t-s+1$ in the bipartite graph $(U'_i,U'_{i+1})$. Since $t-s+1$ is odd, we may assume without loss of generality that
	$v_1 \in U'_i$ and $v_{t-s+2} \in U'_{i+1}$. By our assumption, the edge $(v_1,v_2)$ is on some $(U_1,\dots,U_s)$-path, implying that there is a path
	$P' \subseteq U_1 \cup \dots \cup U_i$
	between\footnote{It might be the case that $v_1 \in U_1$ (if $i=1$), in which case $P'$ has no edges.} $U_1$ and $v_1$. Similarly, since the edge $(v_{t-s+1},v_{t-s+2})$ is on some $(U_1,\dots,U_s)$-path, there is a path
	$P'' \subseteq U_{i+1} \cup \dots \cup U_s$
	between $v_{t-s+2}$ to $U_s$.
	Then $P' P P''$ is a path of length $t-1$ inside $U_1 \cup \dots \cup U_s$ between $U_1$ and $U_s$, in contradiction to our assumption.
\end{proof}
\begin{lemma}\label{lem:paths_bound forbidden_even_cycle}
	Let $s,\ell \geq 2$, let $G$ be an $n$-vertex $C_{2\ell}$-free graph, let
	$\{u_0\},U_1,\dots,U_s \subseteq V(G)$ be pairwise-disjoint vertex-sets, and suppose that
	$u_0$ is adjacent to every vertex in $U_1$. Then
	$$
	p(U_1,\dots,U_s) \leq
	\begin{cases}
	(\ell-1)^{(s-1)/2}n^{(s-3)/2}\left( |U_1||U_s| + (\ell-1)n \right)
	& s \text{ is odd},
	\\
	(\ell-1)^{s/2}n^{s/2-1}(|U_1| + |U_2|)
	& s \text{ is even}.
	\end{cases}
	$$
\end{lemma}
\begin{proof}
	It is sufficient to show that the conditions of Lemma \ref{lem:paths_bound main} hold with
	$\lambda = \ell-1 \geq 1$. If $e(U_1,U_2) > (\ell-1)(|U_1| + |U_2|)$ then by Theorem \ref{thm:EG}, there is a path of length $2\ell-1$ in the bipartite graph $(U_1,U_2)$. This path contains a subpath of length $2\ell-2$ with both endpoints in $U_1$, which closes a $2\ell$-cycle with $u_0$, in contradiction to the assumption of the lemma.
	Similarly, if
	$e(N_{U_{i+1}}(u_i),U_{i+2}) >
	(\ell - 1)(\left| N_{U_{i+1}}(u_i) \right| + \left| U_{i+2} \right|)$ for some
	$1 \leq i \leq s-2$ and $u_i \in U_i$, then by Theorem \ref{thm:EG} there is a path of length $2\ell-1$ in the bipartite graph with sides $N_{U_{i+1}}(u_i)$ and $U_{i+2}$. This path contains a subpath of length $2\ell-2$ with both endpoints in $N_{U_{i+1}}(u_i)$, which closes a $2\ell$-cycle with $u_i$, in contradiction to the assumption of the lemma.
\end{proof}
The construction in Claim \ref{prop:blowup_construction_cycles} shows that the bounds in the above two lemmas, as well as in Lemma \ref{lem:paths_bound main}, are tight (up to the constants depending on the parameters $\lambda,s,t,\ell$).
We now derive the following corollary of the above two lemmas, which will be used later on.
\begin{lemma}\label{lem:forbidden_cycles_main}
Let $k,\ell \geq 2$, let $G$ be an $n$-vertex graph and assume either that $G$ is $C_{2\ell}$-free or that $G$ is $C_{2\ell + 1}$-free and $\ell > k$.
Then for every partition $V(G) = V_1 \cup \dots \cup V_{2k+1}$ we have
$$
c(V_1,\dots,V_{2k+1}) \leq
\ell^{k-1} n^{k-2} \cdot
\left[ p(V_1,V_2,V_3,V_4) + p(V_{2k+1},V_1,V_2,V_3) \right].
$$
\end{lemma}
\begin{proof}
Fix any $(V_1,V_2,V_3)$-path $v_1,v_2,v_3$.
We claim that
\begin{equation}\label{eq:2k-3_paths_bound}
p\left( N_{V_4}(v_3),V_5,\dots,V_{2k},N_{V_{2k+1}}(v_1) \right) \leq
\ell^{k-1} n^{k-2} \cdot \left( |N_{V_4}(v_3)| + |N_{V_{2k+1}}(v_1)| \right).
\end{equation}
Indeed, if $G$ is $C_{2\ell}$-free then \eqref{eq:2k-3_paths_bound} follows from Lemma
\ref{lem:paths_bound forbidden_even_cycle}, applied with
$s = 2k-2$, $u_0 = v_3$ and the sets $N_{V_4}(v_3),V_5,\dots,V_{2k},N_{V_{2k+1}}(v_1)$ as $U_1,\dots,U_s$.
If $G$ is $C_{2\ell + 1}$-free and $\ell > k$ then there is no path of length $2\ell - 3$ inside
$V_4 \cup \dots \cup V_{2k+1}$ between a vertex in $N_{V_4}(v_3)$ and a vertex in $N_{V_{2k+1}}(v_1)$, as such a path would close a $(2\ell + 1)$-cycle with the path $v_1v_2v_3$. So \eqref{eq:2k-3_paths_bound} follows from Lemma \ref{lem:paths_bound forbidden_path}, applied with 
$s = 2k-2$, $t = 2\ell-2$, and the sets $N_{V_4}(v_3),V_5,\dots,V_{2k},N_{V_{2k+1}}(v_1)$ as $U_1,\dots,U_s$.
By summing \eqref{eq:2k-3_paths_bound} over all $(V_1,V_2,V_3)$-paths we get
\begin{align*}
c(V_1,\dots,V_{2k+1}) &=
\sum_{v_1,v_2,v_3}{c(v_1,v_2,v_3,V_4,\dots,V_{2k+1})} =
\sum_{v_1,v_2,v_3}{p(N_{V_4}(v_3),V_5,\dots,V_{2k},N_{V_{2k+1}}(v_1))}
\\ &\leq
\ell^{k-1} n^{k-2} \cdot
\sum_{v_1,v_2,v_3}{ \left( |N_{V_4}(v_3)| + |N_{V_{2k+1}}(v_1)| \right) }
\\ &=
\ell^{k-1} n^{k-2} \cdot
\left[ p(V_1,V_2,V_3,V_4) + p(V_{2k+1},V_1,V_2,V_3) \right]\;,
\end{align*}
thus completing the proof.
\end{proof}

\subsection{Proof of Theorem 1}\label{subsec:main}

Here we prove Theorem \ref{thm:main}.
The proof is split into several parts: Lemma \ref{lem:ex_even_cycles} handles the case that both cycle lengths are even; Lemma \ref{lem:ex_odd_even_cycles} handles the case where the forbidden cycle is even and the cycle whose number of copies is maximized is odd; finally, Lemma \ref{lem:ex_odd_cycles} handles the case where the cycle lengths are non-consecutive odd integers.
For convenience, we rephrase each of the cases, denoting the cycle lengths by $2k$ or $2k+1$ and $2\ell$ or $2\ell+1$ (rather than $k$ and $\ell$).

\begin{lemma}\label{lem:ex_even_cycles}
For every $k,\ell \geq 2$ we have $\ex(n,C_{2k},C_{2\ell})=O_{k}(\ell^k n^k)$.
\end{lemma}
\begin{proof}
Let $G$ be an $n$-vertex $C_{2\ell}$-free graph. By Claim \ref{prop:random_partition}, it is enough to prove that
$c(V_1,\dots,V_{2k}) = O(\ell^k n^k)$ for every partition
$V(G) = V_1 \cup \dots \cup V_{2k}$.
Consider one such partition. Fixing $v_1 \in V_1$, we apply Lemma
\ref{lem:paths_bound forbidden_even_cycle} with $s = 2k-1$, $u_0 = v_1$ and the sets $N_{V_2}(v_1),V_3,\dots,V_{2k-1},N_{V_{2k}}(v_1)$ as $U_1,\dots,U_s$, to get
\begin{equation*}
c(v_1,V_2,\dots,V_{2k}) =
p(N_{V_2}(v_1),V_3,\dots,V_{2k-1},N_{V_{2k}}(v_1))
\leq
\ell^{k-1}n^{k-2} \cdot
\left( |N_{V_2}(v_1)| \cdot |N_{V_{2k}}(v_1)| + \ell n \right).
\end{equation*}
By summing over all $v_1 \in V_1$, we get
\begin{align*}
c(V_1,\dots,V_{2k}) &= \sum_{v_1 \in V_1}{c(v_1,V_2,\dots,V_{2k})} \leq
\ell^{k-1}n^{k-2} \cdot
\left( \sum_{v_1 \in V_1}{|N_{V_2}(v_1)| \cdot |N_{V_{2k}}(v_1)|} \right) +
\ell^k n^{k-1} \cdot |V_1| \\ &=
\ell^{k-1}n^{k-2} \cdot p(V_{2k},V_1,V_2) + \ell^k n^{k-1} \cdot |V_1| = O(\ell^k n^k) \;,
\end{align*}
where in the last inequality we used Theorem \ref{thm:ex_path_cycle}, which gives
$p(V_{2k},V_1,V_2) = O(\ell n^2)$.
\end{proof}

\begin{lemma}\label{lem:ex_odd_even_cycles}
For every $k \geq 2$ we have
\begin{equation*}
\ex(n,C_{2k+1},C_{2\ell}) \leq
\begin{cases}
O_k(n^{k+1/2}) & \ell = 2, \\
O_k(\ell^{k+1}n^k) & \ell \geq 3.
\end{cases}
\end{equation*}
\end{lemma}
\begin{proof}
We start with the case that $\ell \geq 3$.
Let $G$ be an $n$-vertex $C_{2\ell}$-free graph. By Claim \ref{prop:random_partition}, it is enough to prove that for every partition $V(G) = V_1 \cup \dots \cup V_{2k+1}$ we have
$c(V_1,\dots,V_{2k+1}) = O(\ell^{k+1}n^k)$.
By Theorem \ref{thm:ex_path_cycle} we have
$p(V_{2k+1},V_1,V_2,V_3),p(V_1,V_2,V_3,V_4) \leq O(\ell^2 n^2)$. Plugging these estimates into Lemma \ref{lem:forbidden_cycles_main} gives
$c(V_1,\dots,V_{2k+1}) =  O(\ell^{k+1}n^k)$, as required.

The proof for the case $\ell = 2$ is similar. As in the previous case, we consider a partition 
$V(G) = V_1 \cup \dots \cup V_{2k+1}$ of an $n$-vertex $C_4$-free graph. The only difference is that for $\ell = 2$, Theorem \ref{thm:ex_path_cycle} gives
$p(V_{2k+1},V_1,V_2,V_3),p(V_1,V_2,V_3,V_4) = O(n^{5/2})$. Plugging this into Lemma \ref{lem:forbidden_cycles_main} gives the required bound $c(V_1,\dots,V_{2k+1}) = O_k(n^{k+1/2})$.
\end{proof}

\begin{lemma}\label{lem:ex_odd_cycles}
For every $2 \leq k < \ell - 1$ we have
$\ex(n,C_{2k+1},C_{2\ell+1}) = O( (2k+1)^{2k} \ell^{k+1} n^k )$.
\end{lemma}
\begin{proof}
	Let $G$ be an $n$-vertex $C_{2\ell+1}$-free graph.
	By Claim \ref{prop:random_partition}, we only need to prove that the bound
	$c(V_1,\dots,V_{2k+1}) \leq O( \ell^{k+1} n^k )$ holds for every partition
	$V(G) = V_1 \cup \dots \cup V_{2k+1}$.
	Fix one such partition.
	We may and will assume that for every $1 \leq i \leq 2k+1$, every edge in $E(V_i,V_{i+1})$ is on some $(V_1,\dots,V_{2k+1})$-cycle.
	We claim that the bipartite graph $(V_i,V_{i+1})$ is $C_{2\ell - 2k + 2}$-free for every
	$1 \leq i \leq 2k+1$ (with indices taken modulo $2k+1$). Assume by contradiction that there is a $(2\ell - 2k + 2)$-cycle $C$ in the bipartite graph $(V_i,V_{i+1})$, and let $e \in E(V_i,V_{i+1})$ be an arbitrary edge of $C$. By our assumption, there is a $(V_1,\dots,V_{2k+1})$-cycle $C'$ containing $e$. But now $C \cup C' \setminus \{e\}$ is a $(2\ell+1)$-cycle, a contradiction.
	
	In light of the above, we may apply Lemma \ref{lem:P3 even_cycle main} to $(V_{2k+1},V_1,V_2,V_3)$ with $\ell - k + 1 \geq 3$ in place of $\ell$ and thus obtain subsets
	$V'_1 \subseteq V_1$, $V'_2 \subseteq V_2$ satisfying
	$e(V'_1,V_{2k+1}),e(V'_1,V_2),e(V'_2,V_1),e(V'_2,V_3) = O(\ell n)$ and
	$p(V_{2k+1},V_1 \setminus V'_1, V_2 \setminus V'_2, V_3) = O(\ell^2 n^2)$.
	Similarly, applying Lemma \ref{lem:P3 even_cycle main} to $V_1,V_2,V_3,V_4$ gives subsets $V''_2 \subseteq V_2$ and
	$V''_3 \subseteq V_3$ such that
	$e(V''_2,V_1),e(V''_2,V_3),e(V''_3,V_2),e(V''_3,V_4) = O(\ell n)$ and
	$p(V_1, V_2 \setminus V''_2, V_3 \setminus V''_3, V_4) = O(\ell^2 n^2)$.
	Setting $W_1 = V_1 \setminus V'_1$,
	$W_2 = V_2 \setminus (V'_2 \cup V''_2)$ and $W_3 = V_3 \setminus V''_3$, we see that
	\begin{equation}\label{eq:cycle_bound_five_terms}
	\begin{split}
	& c(V_1,\dots,V_{2k+1}) \leq
	c(W_1,W_2,W_3,V_4,\dots,V_{2k+1}) +
	c(V'_1,V_2,\dots,V_{2k+1}) + \\ & \hspace{1cm}
	c(V_1,V'_2,V_3,\dots,V_{2k+1}) +
	c(V_1,V''_2,V_3,\dots,V_{2k+1}) +
	c(V_1,V_2,V''_3,V_4,\dots,V_{2k+1}).
	\end{split}
	\end{equation}
By our choice of $V'_1,V'_2,V''_2,V''_3$ via Lemma \ref{lem:P3 even_cycle main} and by the definition of the sets $W_1,W_2,W_3$, we have
$p(V_{2k+1},W_1,W_2,W_3) = O(\ell^2 n^2)$ and
$p(W_1,W_2,W_3,V_4) = O(\ell^2 n^2)$. 
Plugging these bounds into Lemma \ref{lem:forbidden_cycles_main} gives
$$
c(W_1,W_2,W_3,V_4,\dots,V_{2k+1}) \leq \ell^{k-1} n^{k-2} \cdot O(\ell^2 n^2) \leq
O( \ell^{k+1}n^k ).
$$

It remains to bound the other four terms in \eqref{eq:cycle_bound_five_terms}.
Consider the term $c(V'_1,V_2,\dots,V_{2k+1})$. Fixing any $v_1 \in V'_1$, note that there is no path of length $2\ell-1$ inside $V_2 \cup \dots \cup V_{2k+1}$ between a vertex in $N_{V_2}(v_1)$ and a vertex in $N_{V_{2k+1}}(v_1)$, as such a path would close a $(2\ell+1)$-cycle with $v_1$. Thus, we may apply Lemma
\ref{lem:paths_bound forbidden_path}
with $s = 2k$, $t = 2\ell$ and $N_{V_2}(v_1),V_3,\dots,V_{2k},N_{V_{2k+1}}(v_1)$ as $U_1,\dots,U_s$, to get
 $$
 c(v_1,V_2,\dots,V_{2k+1}) = p(N_{V_2}(v_1),V_3,\dots,V_{2k},N_{V_{2k+1}}(v_1)) \leq
\ell^{k}n^{k-1}\left( |N_{V_2}(v_1)| + |N_{V_{2k+1}}(v_1)|\right).
 $$
 By summing the above over all $v_1 \in V'_1$ we obtain
\begin{align*}
c(V'_1,V_2,\dots,V_{2k+1})&= \sum_{v_1 \in V'_1} c(v_1,V_2,\dots,V_{2k+1}) \leq
\ell^k n^{k-1} \cdot
\sum_{v_1 \in V'_1} {\left( |N_{V_2}(v_1)| + |N_{V_{2k+1}}(v_1)|\right) }
\\ &=
\ell^k n^{k-1} \cdot \left(  e(V'_1,V_2) + e(V'_1,V_{2k+1})\right) \leq
O( \ell^{k+1} n^k )
\;,
\end{align*}
where the last equality relies on the guarantees of Lemma
\ref{lem:P3 even_cycle main}. The remaining three terms in \eqref{eq:cycle_bound_five_terms} are shown to be
$O( \ell^{k+1} n^k )$ in the same manner. This completes the proof.
\end{proof}

Having proven Theorem \ref{thm:main}, we summarize
our upper bounds on $\ex(n,C_{2k+1},C_{2\ell+1})$ in Lemma \ref{cor:ex_odd_cycles} below. This lemma will be used in Section \ref{sec:test}. We need the well-known Even Cycle Theorem of Bondy and Simonovits:
\begin{theorem}[\cite{BS}]\label{thm:even_cycle}
For every $\ell \geq 2$ we have
$\ex(n,C_{2\ell}) \leq O(\ell n^{1 + 1/\ell})$.
\end{theorem}
\begin{lemma}\label{cor:ex_odd_cycles}
There is an absolute constant $c$ such that for every $1 \leq k < \ell$ we have the following.
\begin{equation*}
\ex(n,C_{2k+1},C_{2\ell+1}) \leq
\begin{cases}
c \ell^2 n^{1+1/\ell} & k = 1, \\
c (2k+1)^{2k} (2\ell+1)^{k+1} n^k & k \geq 2.
\end{cases}
\end{equation*}
\end{lemma}
\begin{proof}
The case $k=1$ follows immediately by combining \eqref{eq:C3} with Theorem \ref{thm:even_cycle}. As for the case $k \geq 2$,
recall that by Lemma \ref{lem:ex_odd_cycles} we have
$\ex(n,C_{2k+1},C_{2\ell+1}) \leq O( (2k+1)^{2k} (2\ell+1)^{k+1} n^k )$
for every $2 \leq k < \ell - 1$. In light of Lemma \ref{lem:2k+1, 2k+3}, this bound holds for
$\ell = k+1$ as well (as $2\ell+1 = 2k+3 > 4$).
\end{proof}

\subsection{Proof of Theorem \ref{thm:ex_path_cycle} and Proposition \ref{prop:C3_even_cycle}}\label{subsec:path_thm_proof}
\noindent
Here we prove Theorem \ref{thm:ex_path_cycle} and Proposition \ref{prop:C3_even_cycle}. For Theorem \ref{thm:ex_path_cycle} we will need the following lemma. 
\begin{lemma}\label{lem:P2_even_cycle}
	Let $\ell \geq 2$ and let $G$ be an $n$-vertex $C_{2\ell}$-free graph.
	Then every $v \in V(G)$ is the endpoint of at most $4(\ell - 1) n$ paths of length 2.
\end{lemma}
\begin{proof}
	Let $v \in V(G)$ and assume, by contradiction, that $v$ is the endpoint of
	$r > 4(\ell - 1) n$ paths of length $2$. Let
	$V(G) \setminus \{v\} = V_1 \cup V_2$ be a random partition, obtained by putting each
	$u \in V(G) \setminus \{v\}$ in one of the sets $V_1,V_2$ with probability $\frac{1}{2}$, independently. Since
	$\mathbb{E}\left[ p(v,V_1,V_2) \right] = \frac{1}{4}r$, there is a choice of $V_1,V_2$ for which
	$e(N_{V_1}(v),V_2) = p(v,V_1,V_2) \geq \frac{1}{4}r > (\ell - 1)n >
	(\ell - 1)\left( |N_{V_1}(v)| + |V_2|\right)$. This stands in contradiction to Lemma \ref{lem:paths_bound forbidden_even_cycle}, applied with
	$s = 2, u_0 = v, U_1 = N_{V_1}(v), U_2 = V_2$.
\end{proof}
\begin{proof}[Proof of Theorem \ref{thm:ex_path_cycle}]
	The lower bounds are proved in Section \ref{sec:lower}: the lower bound for $\ell=2$ is given by Lemma \ref{lem:C4}, and the lower bound for $\ell \geq 3$ is given by Corollary \ref{cor:lower_bounds}.
	Thus, it remains to prove the upper bounds. We prove both cases simultaneously by induction on $k$. The base cases are $k = 2,3$. For $k = 2$, Lemma \ref{lem:P2_even_cycle} implies that
	$\ex(n,P_2,C_{2\ell}) = O(\ell n^2)$, as required.
	
	Suppose now that $k=3$.
	We first handle the case $\ell \geq 3$.
	By Claim \ref{prop:random_partition}, it is enough to show that $p(X,Y,Z,W) \leq O(\ell^2 n^2)$ for every vertex-partition
	$X \cup Y \cup Z \cup W$ of an $n$-vertex $C_{2\ell}$-free graph.
	Let $Y' \subseteq Y$ and $Z' \subseteq Z$ be as in Lemma \ref{lem:P3 even_cycle main}.
	In light of Item 2 in Lemma \ref{lem:P3 even_cycle main}, it is enough to prove that
	$p(X,Y',Z,W) = O(\ell^2 n^2)$ and $p(X,Y,Z',W) = O(\ell^2 n^2)$. Fix any $y \in Y'$. By Lemma \ref{lem:P2_even_cycle}, we have
	$p(y,Z,W) = O(\ell n)$, and hence
	$p(X,y,Z,W) \leq O(\ell n) \cdot |N_X(y)|$. By summing over all $y \in Y'$ and using the guarantees of Lemma \ref{lem:P3 even_cycle main}, we get
	\begin{equation*}
	p(X,Y',Z,W) = \sum_{y \in Y'}{p(X,y,Z,W)} \leq
	O(\ell n) \cdot \sum_{y \in Y'}{|N_X(y)|} = O(\ell n) \cdot e(Y',X) \leq O(\ell^2 n^2).
	\end{equation*}
	The bound $p(X,Y,Z',W) = O(\ell^2 n^2)$ is proven similarly.
	
	Now we handle the case $\ell = 2$. Let $G$ be an $n$-vertex $C_4$-free graph.
	Observe that the number of paths of length $3$ in a graph $G$ is at most $\sum_{v \in V(G)}{\#P_1(v) \cdot \#P_2(v)}$, where $\#P_i(v)$ is the number of paths of length $i$ having $v$ as an endpoint (so $\#P_1(v)$ is just the degree of $v$). By combining Lemma \ref{lem:P2_even_cycle} with Theorem \ref{thm:even_cycle}, we get that
	$\sum_{v \in V(G)}{\#P_1(v) \cdot \#P_2(v)}
	\leq O(n) \cdot 2e(G) \leq O(n^{5/2})$,
	as required.
	
	Let now $k \geq 4$. Let $G$ be an $n$-vertex $C_{2\ell}$-free graph, and observe that the number of paths of length $k$ in $G$ is at most
	$$
	\sum_{v \in V(G)}{\#P_{k-2}(v) \cdot \#P_2(v)} \leq
	O(\ell n) \sum_{v \in V(G)}{\#P_{k-2}(v)} \leq
	O(\ell n) \cdot \ex(n,P_{k-2},C_{2\ell}),
	$$
	where in the first inequality we used Lemma \ref{lem:P2_even_cycle}. Thus,
	$\ex(n,P_k,C_{2\ell}) \leq O(\ell n) \cdot \ex(n,P_{k-2},C_{2\ell})$.
	It is now easy to see that the theorem follows by induction on $k$, with the base cases $k = 2,3$.
\end{proof}
\begin{proof}[Proof of Proposition \ref{prop:C3_even_cycle}]
We start with the lower bound. For $\ell \geq 3$, this is the statement of Claim \ref{prop:C3_even_cycle_lower_bound}.
For $\ell = 2$, we get it from Lemma \ref{lem:C4} and the well-known fact that $\ex(n,C_4) = O(n^{3/2})$ (see Theorem \ref{thm:even_cycle}).
For the upper bound, let $G$ be an $n$-vertex $C_{2\ell}$-free graph, and observe that for every $v \in V(G)$, the neighbourhood of $v$ does not contain a path of length $2\ell - 2$; indeed, such a path would close a copy of $C_{2\ell}$ with $v$. By Theorem \ref{thm:EG} we have $e(N(v)) \leq \frac{2\ell - 3}{2} \cdot |N(v)|$. On the other hand, the number of triangles containing $v$ is exactly $e(N(v))$, so the number of triangles in $G$ is
$$
\frac{1}{3}\sum_{v \in V(G)}{e(N(v))} \leq
\frac{2\ell - 3}{6}\sum_{v \in V(G)}{|N(v)|} = \frac{2\ell - 3}{3} \cdot e(G)
\leq \frac{2\ell - 3}{3} \cdot \ex(n,C_{2\ell})\;,
$$
thus completing the proof.
\end{proof}

\section{Applications of Theorem \ref{thm:main}}\label{sec:test}
In this section we prove Theorems \ref{thm:hierarchy},  \ref{thm:cycle_sequence_charac}  and \ref{thm:two_sided_test}.
Given a monotone graph property $\mathcal{P}$ and $\varepsilon \in (0,1)$, recall that $w_{\mathcal{P}}(\varepsilon)$ is the minimal positive integer such that for every sufficiently large graph $G$ which is $\varepsilon$-far from satisfying $\mathcal{P}$, a randomly-chosen induced subgraph of $G$ of order $w_{\mathcal{P}}(\varepsilon)$ does not satisfy $\mathcal{P}$ with probability at least $\frac{2}{3}$. Recall that for a set of integers $L$, we say
that a graph is $L$-free if it is $C_{\ell}$-free for every $\ell \in L$. In this section we will only consider cases where $L$ is an infinite set of odd integers. Recall that when writing $w_{L}(\varepsilon)$ we interpret $L$ as the {\em property} of being $L$-free.
In what follows, $c,c',c'',c_1,c_2,\dots$ are absolute constants which are implicitly assumed to be large enough.

The following theorem is a special case of the main result of Alon et al. \cite{ADKK}. For a graph $G$, denote by $\maxcut(G)$ the largest size of a cut in $G$.
\begin{theorem}\label{thm:ADKK}
For every
$\varepsilon \in (0,1/2)$, for every $n$-vertex graph $G$ and for every
$q \geq c\varepsilon^{-4}\log(1/\varepsilon)$,
a uniformly chosen set $Q \in \binom{V(G)}{q}$
satisfies
$
\left| \frac{\maxcut(G)}{n^2} - \frac{\maxcut(G[Q])}{q^2} \right| < \varepsilon
$
with probability at least $\frac{5}{6}$.
\end{theorem}

\noindent
We now derive the following lemma from Theorem \ref{thm:ADKK}.

\begin{lemma}\label{thm:bipartite_distance}
For every $\varepsilon \in (0,1)$ and for every graph $G$ which is $\varepsilon$-far from bipartiteness, it holds that with probability at least $\frac{2}{3}$, a random induced subgraph of $G$ of order
$c \varepsilon^{-5}$ is $\frac{\varepsilon}{2}$-far from bipartiteness.
\end{lemma}
\begin{proof}
Let $G$ be a graph which is $\varepsilon$-far from bipartiteness. Then clearly
$$\maxcut(G) \leq e(G) - \varepsilon n^2 =
\left( \frac{e(G)}{n^2} - \varepsilon \right)n^2.$$
Set $q = c \varepsilon^{-5}$ and let $Q \in \binom{V(G)}{q}$ be chosen uniformly at random.
Then with probability at least $\frac{5}{6}$ we have
$\maxcut(G[Q]) \leq (\frac{e(G)}{n^2} - \frac{3\varepsilon}{4})q^2$,
where we applied Theorem \ref{thm:ADKK} with $\frac{\varepsilon}{4}$ in place of $\varepsilon$. By a standard second-moment-method argument one can easily show that a randomly chosen induced subgraph of order at least $c \varepsilon^{-2}$ has the same edge density as $G$, up to an additive error of $\varepsilon$. Thus, (by applying this argument with $\varepsilon/4$ in place of $\varepsilon$), the inequality
$$
\left| \frac{e(G)}{n^2} - \frac{e(G[Q])}{q^2} \right| < \frac{\varepsilon}{4}
$$
holds with probability at least $\frac{5}{6}$. Thus, with probability at least $\frac{2}{3}$ we have
$$\maxcut(G[Q]) \leq \left( \frac{e(G)}{n^2} - \frac{3\varepsilon}{4} \right)q^2 \leq
e(G[Q]) - \frac{\varepsilon}{2}q^2,$$
which implies that $G[Q]$ is $\frac{\varepsilon}{2}$-far from bipartiteness. This completes the proof.
\end{proof}
The next lemma we will need is Lemma \ref{cor:k-cycle_sample} below, which relies on Lemma \ref{thm:bipartite_distance} and on the following theorem of Koml\'{o}s.
\begin{theorem}[\cite{Komlos}]\label{thm:Komlos}
For every $\varepsilon \in (0,1/2)$, every graph which is $\varepsilon$-far from bipartiteness contains an odd cycle of length at most
$c \varepsilon^{-1/2}$.
\end{theorem}
\begin{lemma}\label{cor:k-cycle_sample}
Let
$\varepsilon \in (0,1)$, suppose that
$n \geq q \geq c_1 \varepsilon^{-11}$
and let $G$ be an $n$-vertex graph. If $G$ is $\varepsilon$-far from being bipartite then there is an odd $3 \leq s \leq c_1 \varepsilon^{-1/2}$ such that with probability at least $\frac{2}{3}$, a random induced subgraph of $G$ of order $q$ contains at least
$\left( \varepsilon^{6} q/c_1 \right)^{s}$
copies of $C_{s}$.
\end{lemma}
\begin{proof}
	By Theorem \ref{thm:bipartite_distance}, a uniformly chosen
	$P \in \binom{V(G)}{c\varepsilon^{-5}}$ induces a graph which is
	$\frac{\varepsilon}{2}$-far from bipartiteness with probability at least $2/3$. By Theorem \ref{thm:Komlos}, such an induced subgraph contains an odd cycle of length at most
	$c (\varepsilon/2)^{-1/2}$.
	Thus, there is
	$3 \leq s \leq c (\varepsilon/2)^{-1/2}$
	such that a random
	$P$ as above
	contains an $s$-cycle with probability at least
	$\varepsilon^{1/2}/c'$.
	Set
	$d = 4c'\varepsilon^{-1/2}$
	and let $P_1,\dots,P_d \in \binom{V(G)}{c\varepsilon^{-5}}$ be chosen uniformly at random and independently. Setting
	$R = P_1 \cup \dots \cup P_d$, we see that $G[R]$ contains an $s$-cycle with probability at least
	$1 -
	\left( 1 - \varepsilon^{1/2}/c' \right)^{4c'\varepsilon^{-1/2}}
	\geq 1 - e^{-4} \geq 11/12$.
	Moreover, the probability that there are $1 \leq i < j \leq d$ for which $P_i \cap P_j \neq \emptyset$ is at most
	$\binom{d}{2} \cdot n \cdot ( c\varepsilon^{-5}/n )^{2} \leq
	c'' \varepsilon^{-11}/n \leq \frac{1}{2}$,
	where in the last inequality we used the assumption that
	$n \geq c_1 \varepsilon^{-11}$. Thus, setting
	$r = d \cdot c\varepsilon^{-5} = 4cc' \varepsilon^{-11/2}$,
	we see that
	$\mathbb{P}[|R| = r] \geq \frac{1}{2}$.
	Since $G[R]$ contains an $s$-cycle with probability at least $\frac{11}{12}$, we infer that at least a $\frac{5}{6}$-fraction of all sets
	$R' \in \binom{V(G)}{r}$ are such that $G[R']$ contains an $s$-cycle.
	Let $\mathcal{R}$ be the set of all
	$R' \in \binom{V(G)}{r}$ having this property, and note that
	$|\mathcal{R}| \geq \frac{5}{6} \binom{n}{r}$.
	
	Fix any $q \geq r$.
	For $Q \in \binom{V(G)}{q}$, define the random variable
	$Z(Q) = | \binom{Q}{r} \cap \mathcal{R} |$ (namely, $Z(Q)$ is the number of sets in $\mathcal{R}$ which are contained in $Q$), and let
	$
	\mathcal{Q} =
	\left\{Q \in \binom{V(G)}{q} : Z(Q) \geq \frac{1}{2}\binom{q}{r} \right\}.
	$
	By linearity of expectation, we have
	$\mathbb{E}[Z] = |\mathcal{R}| \cdot \binom{q}{r}/\binom{n}{r} \geq \frac{5}{6}\binom{q}{r}$.
	Since
	$0 \leq Z  \leq \binom{q}{r}$,
	it is now easy to deduce (by averaging) that
	$\mathbb{P}[Z \geq \frac{1}{2}\binom{q}{r}] \geq \frac{2}{3}$,
	implying that
	$|\mathcal{Q}| \geq \frac{2}{3} \binom{n}{q}$.
	
	Now let $Q \in \mathcal{Q}$. By the definition of $\Q$, there are at least $\frac{1}{2}\binom{q}{r}$
	$r$-sets $R \subseteq Q$ such that $G[R]$ contains a copy of $C_{s}$.
	On the other hand, a copy of $C_{s}$ in $G[Q]$ is contained in exactly $\binom{q-s}{r-s}$ such $r$-sets. Thus, $G[Q]$ contains at least
	\begin{equation*}
	\frac{\binom{q}{r}}{2\binom{q-s}{r-s}} =
	\frac{\binom{q}{s}}{2\binom{r}{s}} \geq
	\frac{1}{2}\left( \frac{q}{er} \right)^{s} \geq
	\big( \varepsilon^6 q/c_1 \big)^{s} \; .
	\end{equation*}
	copies of $C_{s}$. This completes the proof.
\end{proof}
Lemma \ref{lem:far_from_bipartite_many_L_cycles}, stated below, is the main lemma in this section. Its proof uses Lemma \ref{cor:ex_odd_cycles}, Lemma \ref{cor:k-cycle_sample} and the following lemma from \cite{AS_separation}.
\begin{lemma}[\cite{AS_separation}]\label{lem:blowup_distance}
Let $K$ be a $k$-vertex graph, let $F$ be an $f$-vertex graph which has a homomorphism into $K$ and let $G$ be the $\frac{n}{k}$-blowup of $K$ where $n \geq n_0(k,f)$. Then $G$ is $\frac{1}{2k^2}$-far from being $F$-free.
\end{lemma}
\begin{lemma}\label{lem:far_from_bipartite_many_L_cycles}
There is a constant $c_2 \geq c_1$ (where $c_1$ is from Lemma \ref{cor:k-cycle_sample}) such that the following holds. Let $(\ell_i)_{i \geq 1}$ be an infinite increasing sequence of odd integers with $\ell_1 \geq 3$, and set $L = \{\ell_i : i \geq 1\}$. Then the following holds.
\begin{enumerate}
\item Let $\varepsilon \in (0,1)$ be small enough so that $c_1 \varepsilon^{-1/2} \geq \ell_1$. Let $\ell_i$ be the maximal element of $L$ not larger than $c_1 \varepsilon^{-1/2}$, let
$n \geq q \geq c_2 \varepsilon^{-13} \cdot \ell_1^{2} \cdot \ell_{i+1}$, and let $G$ be an $n$-vertex graph which is $\varepsilon$-far from being bipartite. Then with probability at least $\frac{2}{3}$, a random induced subgraph of $G$ of order $q$ is not $L$-free. Thus,
$w_{L}(\varepsilon) \leq c_2 \varepsilon^{-13} \cdot \ell_1^{2} \cdot \ell_{i+1}$.
\item For every $i \geq 1$ we have $w_{L}( \frac{1}{2(\ell_i + 2)^2} ) \geq \ell_{i+1}$.
\end{enumerate}
\end{lemma}
\begin{proof}
	We start by proving the first assertion of Item 1. Let $G$ be an $n$-vertex graph which is $\varepsilon$-far from bipartiteness.
	By Lemma \ref{cor:k-cycle_sample}, there is an odd
	$3 \leq s \leq c_1 \varepsilon^{-1/2}$ such that for a randomly chosen $Q \in \binom{V(G)}{q}$, the graph $G[Q]$ contains at least
	$( \varepsilon^{6} q/c_1 )^{s}$ copies of $C_{s}$ with probability at least $\frac{2}{3}$. We claim that if $G[Q]$ has this property then $G[Q]$ is not $L$-free. This will show that a random induced subgraph of $G$ of order $q$ is not $L$-free with probability at least $\frac{2}{3}$. This will also prove the upper bound on $w_L(\varepsilon)$ stated in Item 1, since every graph which is $\varepsilon$-far from being $L$-free is also $\varepsilon$-far from bipartiteness (as $L$ contains only odd integers).
	
	Assume first that $s=3$. If $\ell_1 = 3$ then $G[Q]$ is clearly not $L$-free, as it contains at least one triangle. So we may assume that $\ell_1 = 2\ell + 1 > 3$. It is easy to see that for $c_2$ large enough, our choice of $q$ guarantees that
	$$(\varepsilon^{6} q/c_1)^3 >
	c \ell_1^2 q^{3/2} >
	c \ell^2 q^{1+1/\ell} \geq
	\ex(q,C_3,C_{2\ell+1})\;,$$ where in the last inequality we use Lemma \ref{cor:ex_odd_cycles}. This means that $G[Q]$ contains more triangles than $\ex(q,C_3,C_{2\ell + 1})$. So $G[Q]$ contains a cycle of length $\ell_1 = 2\ell+1$ and hence is not $L$-free.
	
	Assume from now on that $s > 3$. Observe that for a large enough $c_2$ we have
	$$
	( \varepsilon^{6} q/c_1 )^{s} >
	c \cdot (c_1 \varepsilon^{-1/2})^s \cdot \ell_{i+1}^{s/2} \cdot q^{s/2} \geq
	c s^s \ell_{i+1}^{s/2} q^{s/2} \geq
	c s^s \ell_{i+1}^{(s+1)/2} q^{(s-1)/2} \geq
	\ex(q,C_s,C_{\ell_{i+1}})\;,
	$$
	where in the first and third inequalities we use our choice of $q$, in the second inequality we use $s \leq c_1 \varepsilon^{-1/2}$ and in the last inequality we use Lemma \ref{cor:ex_odd_cycles} with $2k+1 = s$ and $2\ell+1 = \ell_{i+1}$, noting that
	$s < \ell_{i+1}$ by our choice of $\ell_{i}$ and by $s \leq c_1 \varepsilon^{-1/2}$.
	As $G[Q]$ contains more $s$-cycles than $\ex(q,C_s,C_{\ell_{i+1}})$, it must contain a cycle of length $\ell_{i+1}$. Thus, $G[Q]$ is not $L$-free.
	
	We now prove the second Item. Fixing $i \geq 1$, let $n$ be large enough so that Lemma \ref{lem:blowup_distance} is applicable to $k = \ell_i + 2$ and $f = \ell_{i+1}$, and let $G$ be the $\frac{n}{\ell_i + 2}$-blowup of $C_{\ell_i + 2}$. Note that $C_{\ell_{i+1}}$ has a homomorphism into $C_{\ell_i + 2}$, as
	$\ell_{i+1} \geq \ell_i + 2$. Thus, by applying Lemma \ref{lem:blowup_distance} with
	$K = C_{\ell_i + 2}$ and $F = C_{\ell_{i+1}}$, we conclude that $G$ is
	$\frac{1}{2(\ell_i + 2)^2}$-far from being $C_{\ell_{i+1}}$-free and hence also $\frac{1}{2(\ell_i + 2)^2}$-far from being $L$-free. On the other hand, there is no homomorphism from $C_{k}$ to
	$C_{\ell_i + 2}$ for any odd $k \leq \ell_i$. Thus, every subgraph of $G$ on less than $\ell_{i+1}$ vertices is $L$-free. Item 2 of the lemma follows.
\end{proof}

\noindent
The proofs of Theorems \ref{thm:hierarchy} and \ref{thm:cycle_sequence_charac} now follow quite easily from the above lemma.

\begin{proof}[Proof of Theorem \ref{thm:hierarchy}]
Set $\ell_1 = 3$ and
$\ell_{i+1} = 2f( \frac{1}{2(\ell_i + 2)^2} ) + 1$. Then $\ell_i$ is odd for every $i \geq 1$, and $(\ell_i)_{i \geq 1}$ is increasing as
$f$ satisfies $f(x) \geq 1/x$.
Setting $L = \{\ell_i : i \geq 1\}$, we will show that the property of $L$-freeness satisfies the assertion of the theorem. More precisely, we will show that there is an absolute constant $\varepsilon_0 > 0$ such that $w_{L}(\varepsilon) \leq \varepsilon^{-14}f(\varepsilon/c)$ for every $\varepsilon < \varepsilon_0$, and that
$w_{L}(\varepsilon) \geq f(\varepsilon)$ for an infinite sequence of values of $\varepsilon$ which tends to $0$.
Let $\varepsilon \in (0,1)$ be small enough so that
$c_1 \varepsilon^{-1/2} \geq 3 = \ell_1$, and let $\ell_i$ be the maximal element of $L$ not larger than $c_1 \varepsilon^{-1/2}$.
Item 1 of Lemma \ref{lem:far_from_bipartite_many_L_cycles} implies that
$$
w_{L}(\varepsilon) \leq
c_2 \varepsilon^{-13} \cdot \ell_1^2 \cdot \ell_{i+1} =
9c_2 \varepsilon^{-13} \cdot \ell_{i+1} \leq
27c_2 \varepsilon^{-13} \cdot f( \frac{1}{2(\ell_i + 2)^2} ) \leq
\varepsilon^{-14} \cdot f(\varepsilon/c)\;,
$$
where in the last inequality we used that
$\ell_i \leq c_1 \varepsilon^{-1/2}$, that $f$ is decreasing, and that $1/\varepsilon > 27c_2$ (which can be guaranteed by appropriately choosing $\varepsilon_0$).
The second part of Lemma \ref{lem:far_from_bipartite_many_L_cycles} implies that for every
$i \geq 1$,
$w_{L}( \frac{1}{2(\ell_i + 2)^2} ) \geq \ell_{i+1} > f( \frac{1}{2(\ell_i + 2)^2} )$. So there is a decreasing sequence $(\varepsilon_i)_{i \geq 1}$ with $\varepsilon_i \rightarrow 0$ (namely $\varepsilon_i = \frac{1}{2(\ell_i + 2)^2}$) such that
$w_L(\varepsilon_i) \geq f(\varepsilon_i)$.
The theorem follows.
\end{proof}

\begin{proof}[Proof of Theorem \ref{thm:cycle_sequence_charac}]
	The first part of Lemma \ref{lem:far_from_bipartite_many_L_cycles} implies that for a sufficiently small $\varepsilon$ we have
	$w_L(\varepsilon) \leq \poly(1/\varepsilon) \cdot \ell_{i+1}$, where $\ell_i$ is the maximal element of $L$ not larger than $c_1 \varepsilon^{-1/2}$. Thus, if $\ell_{i+1} \leq \ell_i^d$ for some $d = d(L)$ and every sufficiently large $i$, then
	$w_L(\varepsilon) \leq \poly(1/\varepsilon)$ for every sufficiently small $\varepsilon$. On the other hand, the second part of Lemma \ref{lem:far_from_bipartite_many_L_cycles} implies that unless $\ell_{i+1} \leq \ell_i^d$ for some $d = d(L)$ and for every large enough $i$, the function $w_L(\varepsilon)$ is super-polynomial in $1/\varepsilon$ for infinitely many values of $\varepsilon$. We conclude that $w_{L}(\varepsilon) = \poly(1/\varepsilon)$ if and only if $\ell_{i+1} \leq \ell_i^d$ for every large enough $i$, which is equivalent to having
	$\limsup_{j \longrightarrow \infty}{\frac{\log \ell_{j+1}}{\log \ell_j}} \leq d < \infty$.
\end{proof}

\begin{lemma}\label{lem:o(1)-close to bipartite}
Let $(\ell_i)_{i \geq 1}$ be an infinite increasing sequence of odd integers with $\ell_1 \geq 3$, and set $L = \{\ell_i : i \geq 1\}$. Then every $L$-free graph is $o(1)$-close to bipartiteness.
\end{lemma}
\begin{proof}
Our goal is to show that for every sufficiently small $\varepsilon$ there is $n_0(\varepsilon)$ such that every $L$-free graph on $n \geq n_0(\varepsilon)$ vertices is $\varepsilon$-close to being bipartite. So fix $\varepsilon > 0$ small enough to satisfy
$c_1 \varepsilon^{-1/2} \geq \ell_1$, and let
$\ell_i$ be the maximal element of $L$ not larger than $c_1 \varepsilon^{-1/2}$. By (the contrapositive of) Item 1 in Lemma \ref{lem:far_from_bipartite_many_L_cycles}, every $n$-vertex $L$-free graph is $\varepsilon$-close to bipartiteness, provided that $n$ is large enough to satisfy
$n \geq c_2 \varepsilon^{-13} \cdot \ell_1^2 \cdot \ell_{i+1}$.
This completes the proof.
\end{proof}

The quantitative version of Lemma \ref{lem:o(1)-close to bipartite} states that $L$-free $n$-vertex graphs are roughly $\Theta(\ell_i^{-2})$-close to bipartiteness, where $i$ is the maximal integer satisfying $n \geq \ell_{i+1}$ (here we assume that the sequence $(\ell_i)_{i \geq 1}$ grows fast enough).
Let us explain why this dependence on the sequence $(\ell_i)_{i \geq 1}$ is unavoidable.
For $n = \ell_{i+1}$, let $G$ be the
$\frac{n-1}{\ell_i + 2}$-blowup of
$C_{\ell_i + 2}$, plus an isolated vertex. Then $G$ is $L$-free; it contains neither an odd cycle of length at most $\ell_i$ (as such a cycle is not homomorphic to $C_{\ell_i + 2}$), nor an odd cycle of length at least $\ell_{i+1}$ (as $\ell_{i+1} > n-1$ and $G$ has an isolated vertex). Nonetheless, it is easy to see that $G$ is $\Theta(\ell_i^{-2})$-far from bipartiteness. This shows that the $o(1)$-term in Lemma \ref{lem:o(1)-close to bipartite} may tend to zero arbitrarily slowly, depending on the family $L$. For example, if $\ell_i = \mbox{tower}(i)$ then
$\ell_i = \log_2(\ell_{i+1})$, so every $L$-free $n$-vertex graph is roughly
$\Theta(\frac{1}{\log^2 n})$-close to bipartiteness, and this is tight.

\begin{proof}[Proof of Theorem \ref{thm:two_sided_test}]
By (the proof of) Theorem \ref{thm:hierarchy}, there is an increasing sequence of odd integers
$L = \{\ell_1 = 3,\ell_2,\ell_3,\dots\}$ such that $w_L(\varepsilon) \geq f(\varepsilon)$. Thus, it remains to present a $2$-sided tester for $L$-freeness which has query complexity $\poly(1/\varepsilon)$.
Our $\varepsilon$-tester works as follows: it samples a random induced subgraph of the input of order $q = q(\varepsilon) = c \varepsilon^{-5}$
and accepts if and only if this subgraph is $\frac{\varepsilon}{2}$-close to bipartiteness. Let us prove that this algorithm is indeed a valid $\varepsilon$-tester for graphs of order
$n \geq n_0(\varepsilon)$, where $n_0(\varepsilon)$ will be (implicitly) chosen later. Let $G$ be an $n$-vertex input graph. If $G$ is $\varepsilon$-far from $L$-freeness then it is also $\varepsilon$-far from bipartiteness, so Lemma \ref{thm:bipartite_distance} implies that with probability at least $\frac{2}{3}$, $G$ is rejected. Assume now that $G$ is $L$-free. By Lemma \ref{lem:o(1)-close to bipartite}, if $n$ is large enough then $G$ is $\frac{\varepsilon}{12}$-close to bipartiteness.
Hence, there is a set $E \subseteq E(G)$ of size
$|E| \leq \frac{\varepsilon}{12}n^2$ such that $G \setminus E$ (the graph obtained from $G$ by deleting the edges in $E$) is bipartite. Let
$Q = \{x_1,\dots,x_q\}$ denote the vertex-set sampled by the tester. The expected number of pairs $1 \leq i < j \leq q$ for which $\{x_i,x_j\} \in E$ is
$\binom{q}{2} \cdot \frac{2|E|}{n(n-1)} \leq \frac{\varepsilon}{6}q^2$. By Markov's inequality, we have
$|E(G[Q]) \cap E| \leq \frac{\varepsilon}{2}q^2$ with probability at least $\frac{2}{3}$. Thus, with probability at least $\frac{2}{3}$, $G[Q]$ is $\frac{\varepsilon}{2}$-close to bipartiteness (as deleting the edges in $E(G[Q]) \cap E$ makes $G[Q]$ bipartite), and $G$ is accepted by the tester.
\end{proof}

\section{Proof of Lemma \ref{lem:P3 even_cycle main}}\label{sec:key}
We will need an upper bound on Zarankiewicz numbers for even cycles, proved by Naor and Verstra\"{e}te \cite{NV}. For integers $n,m \geq 1$ and $\ell \geq 2$, let $z(n,m,C_{2\ell})$ denote the maximal number of edges in a $C_{2\ell}$-free bipartite graph with sides of size $n$ and $m$.
\begin{theorem}[\cite{NV}]\label{thm:Zarankiewicz_even_cycles}
	For $m \leq n$ it holds that
	\begin{equation*}
	z(n,m,C_{2\ell}) \leq
	\begin{cases}
	(2\ell - 3)\left( (nm)^{1/2 + 1/(2\ell)} + 2n \right) &
	\ell \text{ is odd}, \\
	(2\ell - 3)\left( n^{1/2}m^{1/2 + 1/\ell} + 2n \right) &
	\ell \text{ is even}.
	\end{cases}
	\end{equation*}
\end{theorem}
\noindent
The following lemma is an easy corollary of Theorem \ref{thm:Zarankiewicz_even_cycles}.
\begin{lemma}\label{lem:high_degrees}
Let $\ell \geq 2$, let $G$ be an $n$-vertex graph and let
$X,Y \subseteq V(G)$ be disjoint sets such that the bipartite graph $(X,Y)$ is $C_{2\ell}$-free. Let $Y'$ be the set of all vertices in $Y$ having at least $d$ neighbours in $X$. Then
\begin{equation*}
|Y'| \leq
	\begin{cases}
	\max \{(6\ell/d)^{2\ell/(\ell-1)}n^{(\ell + 1)/(\ell - 1)}, 6\ell n/d\} & \ell \text{ is odd}, \\
	\max \{(6\ell/d)^{2\ell/(\ell-2)}n^{\ell/(\ell - 2)}, 6\ell n/d\} &
	\ell \text{ is even and } \ell \geq 4, \\
	2n/(d - n^{1/2}) & \ell = 2 \text{ and } d > n^{1/2}.
	\end{cases}
\end{equation*}
\end{lemma}
\begin{proof}
Note that
\begin{equation}\label{eq:min_degree edge bound}
d |Y'| \leq e(Y',X) \leq z(n,|Y'|,C_{2\ell}).
\end{equation}
Suppose first that $\ell$ is odd. We apply Theorem \ref{thm:Zarankiewicz_even_cycles} with parameter
$m = |Y'|$. If $(|Y'|n)^{1/2 + 1/(2\ell)} \geq n$ then Theorem \ref{thm:Zarankiewicz_even_cycles} gives $z(n,|Y'|,C_{2\ell}) \leq 6\ell (|Y'|n)^{1/2 + 1/(2\ell)}$, and if
$(|Y'|n)^{1/2 + 1/(2\ell)} \leq n$ then Theorem \ref{thm:Zarankiewicz_even_cycles} gives
$z(n,|Y'|,C_{2\ell}) \leq 6\ell n$.
By combining these inequalities with
\eqref{eq:min_degree edge bound} we get that either
$|Y'| \leq (6\ell/d)^{2\ell/(\ell-1)}n^{(\ell + 1)/(\ell - 1)}$
or
$|Y'| \leq 6\ell n / d$, as required.


Suppose now that $\ell$ is even and $\ell \geq 4$.
By Theorem \ref{thm:Zarankiewicz_even_cycles}, we have
$z(n,|Y'|,C_{2\ell}) \leq 6\ell n^{1/2} |Y'|^{1/2 + 1/\ell}$
if
$n^{1/2} |Y'|^{1/2 + 1/\ell} \geq n$
and
$z(n,|Y'|,C_{2\ell}) \leq 6\ell n$ otherwise. By combining these inequalities with \eqref{eq:min_degree edge bound} we get that
either $|Y'| \leq (6\ell/d)^{2\ell/(\ell-2)}n^{\ell/(\ell - 2)}$
or
$|Y'| \leq 6\ell n / d$, as required.

Finally, suppose that $\ell = 2$ and that $d > n^{1/2}$. Theorem \ref{thm:Zarankiewicz_even_cycles} gives $z(n,|Y'|,C_4) \leq n^{1/2}|Y'| + 2n$. By combining this with
\eqref{eq:min_degree edge bound} we get that $|Y'| \leq 2n/(d - n^{1/2})$, as required.
\end{proof}

\noindent
We are now ready to prove Lemma \ref{lem:P3 even_cycle main}.

\begin{proof}[Proof of Lemma \ref{lem:P3 even_cycle main}]
We start by considering the case of even $\ell \geq 4$.
Define the sets $Y' = \{ y \in Y : |N_X(y)| \geq \ell n^{2/(\ell+2)} \}$ and $Z' = \{ z \in Z : |N_W(z)| \geq \ell n^{2/(\ell+2)} \}$. Apply Lemma \ref{lem:high_degrees} with $d = \ell n^{2/(\ell+2)}$ to get
$|Y'|,|Z'| \leq O(n^{\ell/(\ell + 2)})$.
By plugging these bounds into Theorem \ref{thm:Zarankiewicz_even_cycles}, one can check that
$e(Y',X),e(Y',Z),e(Z',Y),e(Z',W) \leq
z(n,O(n^{\ell/(\ell + 2)}),C_{2\ell}) = O(\ell n)$.
Next, note that by the definitions of $Y'$ and $Z'$ we have
\begin{align*}
p(X,Y \setminus Y' ,Z \setminus Z' ,W) &<
e(Y \setminus Y',Z \setminus Z') \cdot
\ell n^{2/(\ell+2)} \cdot \ell n^{2/(\ell+2)} \\ &\leq
z(n,n,C_{2\ell}) \cdot \ell^2 n^{4/(\ell+2)} \leq
O(\ell^3 n^{1 + 1/\ell + 4/(\ell+2)}),
\end{align*}
where in the last inequality we used Theorem \ref{thm:Zarankiewicz_even_cycles}. So if
$\ell^3 n^{1 + 1/\ell + 4/(\ell+2)} \leq \ell^2 n^2$ then we get the required bound $p(X,Y \setminus Y' ,Z \setminus Z' ,W) = O(\ell^2 n^2)$, and the proof is complete (for even $\ell$). Otherwise, we have
$\ell^3 n^{1 + 1/\ell + 4/(\ell+2)} > \ell^2 n^2$ and hence
$n < \ell^{\ell(\ell+2)/(\ell^2-3\ell-2)} =
\ell \cdot \ell^{(5\ell+2)/(\ell^2-3\ell-2)} \leq O(\ell)$. Since
$p(X,Y,Z,W) \leq n^4$, we have
$p(X,Y,Z,W) \leq n^4 = O(\ell^2 n^2)$, and again we are done.

We now consider the case of odd $\ell \geq 3$. 
We define
$Y' = \{ y \in Y : |N_X(y)| \geq \ell n^{2/(\ell+1)} \}$ and
$Z' = \{ z \in Z : |N_W(z)| \geq \ell n^{2/(\ell+1)} \}$. Similarly to the previous case, we apply Lemma \ref{lem:high_degrees} with
$d = \ell n^{2/(\ell+1)}$ to obtain
$|Y'|,|Z'| \leq O(n^{(\ell - 1)/(\ell + 1)})$. We then plug these bounds into Theorem \ref{thm:Zarankiewicz_even_cycles} to get
$e(Y',X),e(Y',Z),e(Z',Y),e(Z',W) \leq
z(n,O(n^{(\ell - 1)/(\ell + 1)}),C_{2\ell}) = O(\ell n)$. 
It remains to bound $p(X,Y \setminus Y' ,Z \setminus Z' ,W)$. Assume first that $\ell \geq 5$. 
By the definitions of $Y'$ and $Z'$ we have
\begin{align*}
p(X,Y \setminus Y' ,Z \setminus Z' ,W) &<
e(Y \setminus Y',Z \setminus Z') \cdot
\ell n^{2/(\ell+1)} \cdot \ell n^{2/(\ell+1)} \\ &\leq
z(n,n,C_{2\ell}) \cdot \ell^2 n^{4/(\ell+1)} \leq
O(\ell^3 n^{1 + 1/\ell + 4/(\ell+1)}),
\end{align*}
where in the last inequality we used Theorem \ref{thm:Zarankiewicz_even_cycles}. If
$\ell^3 n^{1 + 1/\ell + 4/(\ell+1)} \leq \ell^2 n^2$ then by the above we have 
$p(X,Y \setminus Y' ,Z \setminus Z' ,W) = O(\ell^2 n^2)$, as required. Otherwise, we have
$\ell^3 n^{1 + 1/\ell + 4/(\ell+1)} > \ell^2 n^2$ and hence
$n < \ell^{\ell(\ell+1)/(\ell^2-4\ell-1)} =
\ell \cdot \ell^{(5\ell+1)/(\ell^2-4\ell-1)} = O(\ell)$. But then
$p(X,Y,Z,W) \leq n^4 = O(\ell^2 n^2)$, and again we are done.

Thus, it remains to show that $p(X,Y \setminus Y' ,Z \setminus Z' ,W) = O(n^2)$ when $\ell = 3$. Recall that in this case we defined
$Y' = \{ y \in Y : |N_X(y)| \geq 3 n^{1/2} \}$ and similarly
$Z' = \{ z \in Z : |N_W(z)| \geq 3 n^{1/2} \}$.
We need some additional definitions.
Define $Y_{low} = \{y \in Y: |N_X(y)| < n^{1/3}\}$ and similarly
$Z_{low} = \{z \in Z: |N_W(z)| < n^{1/3}\}$. Define
$\mathcal{I} = \{i : \frac{1}{2}n^{1/3} \leq 2^i < 3n^{1/2}\}$, and for each
$i \in \mathcal{I}$ set
$Y_i = \left\{ y \in Y : 2^{i} \leq |N_X(y)| < 2^{i+1} \right\}$ and
$Z_i = \left\{ z \in Z : 2^{i} \leq |N_W(z)| < 2^{i+1} \right\}$. It is immediate from these definitions that
$Y \setminus Y' \subseteq Y_{low} \cup \bigcup_{i \in I}{Y_i}$
and similarly
$Z \setminus Z' \subseteq Z_{low} \cup \bigcup_{i \in I}{Z_i}$.
Note that
\begin{equation*}
p(X,Y_{low},Z_{low},W) <
e(Y_{low},Z_{low}) \cdot n^{1/3} \cdot n^{1/3} \leq
z(n,n,C_6) \cdot n^{2/3} \leq O(n^2),
\end{equation*}
where in the last inequality we used Theorem \ref{thm:Zarankiewicz_even_cycles}.
Hence, in order to finish the proof we need to bound
$p(X,\bigcup_{i \in I}{Y_i},Z_{low},W)$,
$p(X,Y_{low},\bigcup_{i \in I}{Z_i},W)$
and
$p(X,\bigcup_{i \in I}{Y_i},\bigcup_{i \in I}{Z_i},W)$.
We start with the first two terms.
Fix any $i \in \mathcal{I}$.
By Lemma \ref{lem:high_degrees} with $d = 2^i$, we have
$|Y_i| \leq \max\{ 18^{3} \cdot 2^{-3i} \cdot n^2, 18 \cdot 2^{-i} \cdot n\} =
O(2^{-3i} \cdot n^2)$, where we used the fact that
$9 \cdot 2^{-3i} n^2 > 2^{-i}n$, which follows from $2^i < 3n^{1/2}$. So we get
$$
e(Y_i,Z_{low}) \leq z(|Y_i|,n,C_6) \leq
3 \cdot \left( (|Y_i|n)^{2/3} + 2n \right) \leq
3 \cdot \left( O(n^2 2^{-2i}) + 2n \right) \leq
O(n^2 \cdot 2^{-2i}),
$$
where in the second inequality we used Theorem \ref{thm:Zarankiewicz_even_cycles}, and in the last inequality we used
$n^2 \cdot 2^{-2i} > n/9$ which follows from $2^i < 3n^{1/2}$.
Now we have
\begin{align*}
p(X,\bigcup{Y_i},Z_{low},W) &=
\sum_{i \in \mathcal{I}}{p(X,Y_i,Z_{low},W)} <
\sum_{i \in \mathcal{I}}{e(Y_i,Z_{low}) \cdot 2^{i+1} \cdot n^{1/3}} \leq
\sum_{i \in \mathcal{I}}{O(n^2 \cdot 2^{-2i}) \cdot 2^{i+1} \cdot n^{1/3}} \\ &=
O(n^{7/3}) \cdot \sum_{i \in \mathcal{I}}{2^{-i}} \leq
O(n^{7/3}) \cdot \sum_{i : \; 2^i \geq \frac{1}{2}n^{1/3}}{2^{-i}} =
O(n^{7/3}) \cdot O(n^{-1/3}) =
O(n^2),
\end{align*}
where in the first inequality we used the definitions of $Z_{low}$ and $Y_i$, and in the last inequality we used the definition of $\mathcal{I}$.
The bound
$p(X,Y_{low},\bigcup{Z_i},W) = O(n^2)$
is proved similarly.

Finally, we bound $p(X,\bigcup{Y_i},\bigcup{Z_i},W)$. To this end, fix any
$i,j \in \mathcal{I}$.
We showed above that $|Y_i| \leq O(n^2 \cdot 2^{-3i})$. By the same argument we get $|Z_j| \leq O(n^2 \cdot 2^{-3j})$. Thus we have
\begin{align*}
e(Y_i,Z_j) &\leq
z(|Y_i|,|Z_j|,C_6) \leq
3 \cdot \left( (|Y_i||Z_j|)^{2/3} + |Y_i| + |Z_j| \right) \\ &\leq
O(n^{8/3}) \cdot 2^{-2i} \cdot 2^{-2j} + O(n^2) \cdot (2^{-3i} + 2^{-3j})
\leq
O(n^{8/3}) \cdot 2^{-2i} \cdot 2^{-2j},
\end{align*}
where in the second inequality we used Theorem \ref{thm:Zarankiewicz_even_cycles}, and in the last inequality we used the fact that
$18 n^{8/3} \cdot 2^{-2i} \cdot 2^{-2j} \geq \max
\{ n^2 2^{-3i}, n^2 2^{-3j} \}$,
which follows from
$\frac{1}{2} n^{1/3} \leq 2^i,2^j < 3n^{1/2}$. Now we get
\begin{align*}
p(X,\bigcup{Y_i},\bigcup{Z_i},W) &=
\sum_{i,j \in \mathcal{I}}{p(X,Y_i,Z_j,W)} \leq
\sum_{i,j \in \mathcal{I}}{e(Y_i,Z_j) \cdot 2^{i+1} \cdot 2^{j+1}} \\ &\leq
\sum_{i,j \in \mathcal{I}}{O(n^{8/3}) \cdot 2^{-2i} \cdot 2^{-2j} \cdot 2^{i+1} \cdot 2^{j+1}}
=
O(n^{8/3}) \cdot \sum_{i,j \in \mathcal{I}}{2^{-i} \cdot 2^{-j}}
\\ &\leq
O(n^{8/3}) \cdot \sum_{2^i,2^j \geq \frac{1}{2}n^{1/3}}{2^{-i} \cdot 2^{-j}} =
O(n^{8/3}) \cdot O(n^{-2/3}) = O(n^2),
\end{align*}
where in the first inequality we used the definitions of $Y_i$ and $Z_j$, and in the last inequality we used the definition of $\mathcal{I}$. This completes the proof.
\end{proof}

Let us explain why the sets $Y'$ and $Z'$ in Lemma \ref{lem:P3 even_cycle main} are required, (namely, that the statement $p(X,Y,Z,W) = O_{\ell}(n^2)$ is generally false).
Note that by Theorem \ref{thm:Zarankiewicz_even_cycles}, the average degree between the four sets in Lemma
\ref{lem:P3 even_cycle main}
is $O(n^{1/3})$. One might thus guess that
$p(X,Y,Z,W) = O(n\cdot(n^{1/3})^3)=O(n^2)$. To see that this is not the case, we can take $Y$ to be a single vertex connected to all the vertices of $X$ and $Z$, distribute all other vertices equally among $X$, $Z$ and $W$, and take the bipartite graph between $Z,W$ to be
an extremal graph with no $C_{2\ell}$. Although this example satisfies $p(X,Y,Z,W) \gg n^2$, by removing the single vertex of $Y$
we can make sure that $p(X,Y,Z, W) = O(n^2)$. This is precisely what Lemma
\ref{lem:P3 even_cycle main} states.
What we see in the proof of Theorem \ref{thm:ex_path_cycle} is that if one assumes that the
{\em entire} graph is $C_{2\ell}$-free (and not just
the $3$ bipartite graphs between the $4$ sets) then one no longer needs to remove vertices in order to guarantee that $p(X,Y,Z,W) = O_{\ell}(n^2)$.	

Let us note that Lemma \ref{lem:P3 even_cycle main} does not hold for $\ell = 2$. Indeed, in the proof of Lemma \ref{lem:C4} we construct an $n$-vertex $C_4$-free graph, in which every vertex has degree $\Theta(n^{1/2})$ and lies on $\Theta(n^{3/2})$ paths of length $3$. Taking a random vertex partition of this graph into four sets $X,Y,Z,W$, we see that with high probability, every vertex $y \in Y$ (resp. $z \in Z$) has $\Theta(n^{1/2})$ neighbours in $X$ (resp. $W$), and every vertex in the graph lies on $\Theta(n^{3/2})$ $(X,Y,Z,W)$-paths. Suppose now, by contradiction, that the assertion of Lemma \ref{lem:P3 even_cycle main} holds for the sets $X,Y,Z,W$. Since every $y \in Y$ has $\Theta(n^{1/2})$ neighbours in $X$, and since
$e(Y',X) = O(n)$, we must have $|Y'| = O(n^{1/2})$. Similarly, $|Z'| = O(n^{1/2})$. As every vertex lies on $\Theta(n^{3/2})$ $(X,Y,Z,W)$-paths, we have $p(X,Y,Z,W) = \Theta(n^{5/2})$ and
$p(X,Y',Z,W), p(X,Y,Z',W) = O(n^2)$. But this implies that
$p(X,Y \setminus Y', Z \setminus Z', W) = \Theta(n^{5/2})$, in contradiction to the statement of Lemma \ref{lem:P3 even_cycle main}.

\section{Lower Bound on $\mbox{ex}(n,C_k,C_{\ell})$ and
	$\mbox{ex}(n,P_k,C_{\ell})$}\label{sec:lower}

In this section we prove all lower bounds in Theorem \ref{thm:main} and Proposition \ref{prop:C3_even_cycle}. We start with the following two claims, which handle the case where the forbidden cycle is {\em not} $C_4$. Claim \ref{prop:blowup_construction_cycles} gives lower bounds on $\ex(n,C_k,C_{\ell})$ and $\ex(n,P_k,C_{\ell})$ with the correct dependence on $n$, whenever $\ell \neq 4$. To get the correct dependence on $\ell$ for $\ell \gg k$, we need Claim \ref{prop:blowup_construction_general}, which gives a general lower bound for $\ex(n,T,H)$, but is only applicable when $H$ (that is, $C_\ell$) is somewhat larger than $T$ (that is, $C_k$ or $P_k$). To prove the lower bound for all values of $k$ and $\ell \neq 4$, we need to combine these two claims, which is done in Corollary \ref{cor:lower_bounds}.
For a graph $G$, denote by $\alpha(G)$ the independence number of $G$.
\begin{claim}\label{prop:blowup_construction_cycles}
	For a pair of distinct $k \geq 3$ and $4 \neq \ell \geq 3$ we have
	$\ex(n,C_k,C_{\ell}) =
	\Omega_k \big( n^{\lfloor k/2 \rfloor} \big)$.
	For $k \geq 2$ and $4 \neq \ell \geq 3$ we have
	$\ex(n,P_k,C_{\ell}) =
	\Omega_k \big( n^{\lceil (k+1)/2 \rceil} \big)$.
\end{claim}
\begin{proof}
	We start with the first part of the claim.
	Let $I$ be a maximum independent set of the $k$-cycle $1,\dots,k$.
	Replace each $i \in I$ with a vertex-set of size $m$, where different vertices are replaced with disjoint sets and all of these sets are disjoint from $[k] \setminus I$. Edges of $C_k$ are replaced with complete bipartite graphs. In other words, we take a blowup of $C_{k}$ in which vertices $i \in [k] \setminus I$ are not blown up, while vertices $i \in I$ are blown up to size $m$.
	As $|I| = \alpha(C_k) = \lfloor k/2 \rfloor$, the resulting graph has
	$n := \lfloor k/2 \rfloor \cdot m + \lceil k/2 \rceil$ vertices and
	$m^{|I|} = m^{\lfloor k/2 \rfloor} =
	\Omega_k \big( n^{\lfloor k/2 \rfloor} \big)$ copies of $C_{k}$. It is easy to check that this graph is $C_{\ell}$-free by our assumptions that $\ell \neq k$ and
	$\ell \neq 4$.
	
	We now prove the second part of the claim using a similar construction. Let $I$ be a maximum independent set of the path $P_k$ on the vertices $1,\dots,k+1$. Replace each $i \in I$ with a vertex-set of size $m$, where different vertices are replaced with disjoint sets and all of these sets are disjoint from
	$[k+1] \setminus I$. Edges of $P_k$ are replaced with complete bipartite graphs.
	As $|I| = \alpha(P_k) = \lceil (k+1)/2 \rceil$, the resulting graph has
	$n := \lceil (k+1)/2 \rceil \cdot m + \lfloor (k+1)/2 \rfloor$ vertices and
	$m^{|I|} = m^{\lceil (k+1)/2 \rceil} = \Omega_k ( n^{\lceil (k+1)/2 \rceil} )$ copies of $P_k$. It is easy to check that this graph is $C_{\ell}$-free by our assumptions that
	$\ell \neq 4$.
\end{proof}
\begin{claim}\label{prop:blowup_construction_general}
	Let $T,H$ be graphs on $t$ and $h$ vertices, respectively, such that 
	$h - \alpha(H) - 1 \geq t - \alpha(T)$. Then for every
	$n \geq h - \alpha(H) - 1 + \alpha(T)$,
	it holds that
	$\ex(n,T,H) \geq
	\Omega_{t}
	\left( (h-\alpha(H))^{t - \alpha(T)}n^{\alpha(T)} \right)$.
\end{claim}
\begin{proof}
	Suppose that $V(T) = \{1,\dots,t\}$ and let $I$ be a maximum independent set of $T$.
	Let $U_1,\dots,U_t$ be disjoint vertex-sets such that 
	$|U_1| + \dots + |U_t| = n$ and such that the following holds:
	$\sum_{i \in V(T) \setminus I}{|U_i|} = h - \alpha(H) - 1$, these
	$h - \alpha(H) - 1$ vertices are divided as equally as possible among the $t - \alpha(T)$ sets
	$(U_i)_{i \in V(T) \setminus I}$, and the
	$n - h + \alpha(H) + 1$ vertices of $\bigcup_{i \in I}{U_i}$ are divided as equally as possible among $(U_i)_{i \in I}$. Then none of $U_1,\dots,U_t$ is empty by the assumptions of the claim.
	Define a graph $G$ on $U_1 \cup \dots \cup U_t$ by making $(U_i,U_j)$ a complete bipartite graph if $\{i,j\} \in E(T)$, and an empty bipartite graph otherwise (there are no edges inside the sets $U_1,\dots,U_t$).
	Then $G$ has
	$\Omega_{t}
	\left( (h-\alpha(H))^{t - \alpha(T)}n^{\alpha(T)} \right)$
	copies of $T$. It remains to show that $G$ is $H$-free. Assume by contradiction that there is a copy of $H$ in $G$. Then this copy contains two adjacent vertices which are both in $\bigcup_{i \in I}{U_i}$, since
	$\sum_{i \in V(T) \setminus I}{|U_i|} < h - \alpha(H)$. But
	$\bigcup_{i \in I}{U_i}$ is an independent set in $G$, as $I$ is an independent set in $T$ and $G$ is a blowup of $T$, a contradiction.
\end{proof}
We are now ready to prove the lower bounds in the last two items of Theorem \ref{thm:main} and in the second item of Theorem \ref{thm:ex_path_cycle}. In other words, we handle all cases in which the forbidden cycle is not $C_4$. 
\begin{corollary}\label{cor:lower_bounds}
For a pair of distinct $k \geq 3$ and $4 \neq \ell \geq 3$ we have
$\ex(n,C_k,C_{\ell}) =
\Omega_k \big( \ell^{\lceil k/2 \rceil} n^{\lfloor k/2 \rfloor} \big)$.
For $k \geq 2$ and $4 \neq \ell \geq 3$ we have
$\ex(n,P_k,C_{\ell}) =
\Omega_k(\ell^{\lfloor (k+1)/2 \rfloor}n^{\lceil (k+1)/2 \rceil})$.
\end{corollary}
\begin{proof}
Note that since our bound hides constants that depend on $k$, if 
$\ell < k + 3$ then the assertion of the corollary follows from Claim \ref{prop:blowup_construction_cycles}. So we may assume that 
$\ell \geq k + 3$, which implies that 
$\lceil \ell/2 \rceil \geq \lceil k/2 \rceil + 1$. Under this assumption, Claim \ref{prop:blowup_construction_general} is applicable to $(T,H)=(C_k,C_{\ell})$, giving 
$\ex(n,C_k,C_{\ell}) =
\Omega_k(\ell^{\lceil k/2 \rceil}n^{\lfloor k/2 \rfloor})$, 
and to $(T,H) = (P_k,C_{\ell})$, giving 
$\ex(n,P_k,C_{\ell}) =
\Omega_k(\ell^{\lfloor (k+1)/2 \rfloor}n^{\lceil (k+1)/2 \rceil})$.
\end{proof}


When excluding $C_4$, a different construction is required.
The construction we use is due to Erd\H{o}s and R\'{e}nyi \cite{ER}. The case of $\ex(n,C_3,C_4)$ was handled (using the same construction) in \cite{Alon_Shikhelman}. Via the following lemma, we get the lower bound in the first item of Theorem \ref{thm:main} and of Theorem \ref{thm:ex_path_cycle}. 
\begin{lemma}\label{lem:C4}
	Let $q$ be a
	prime power and set
	$n = q^2 - 1$. Then there is an $n$-vertex $C_4$-free graph which contains at least $\left( \frac{1}{2k} - o(1) \right)n^{\frac{k}{2}}$ copies of $C_k$ for every $4 \neq k \geq 3$, and at least
	$\left( \frac{1}{2} - o(1) \right)n^{\frac{k}{2}+1}$ copies of $P_{k}$ for every $k \geq 1$. Here, the $o(1)$ term is a function which depends on $k$ and tends to $0$ as $n$ tends to infinity. Hence,
	$\ex(n,C_k,C_4) \geq \left( \frac{1}{2k} - o(1) \right)n^{\frac{k}{2}}$ for every $4 \neq k \geq 3$, and
	$\ex(n,P_k,C_4) \geq \left( \frac{1}{2} - o(1) \right)n^{\frac{k}{2}+1}$ for every $k \geq 1$.
\end{lemma}
\begin{proof}
	The last part of the theorem is deduced from the first part as follows. It is known that for every large enough $x$ there is a prime in the interval $[x-x^{\theta},x]$ for an absolute constant
	$\theta \in [\frac{1}{2},1)$, see e.g. \cite{BHP}.
	Fixing a large enough $n$, let $p$ be a prime in
	$[x-x^{\theta},x]$ for $x = n^{1/2}$.
	Now take the construction from the first part of the theorem on $p^2-1$ vertices and add isolated vertices to get a graph on $n$ vertices. This graph gives the required lower bounds on $\ex(n,C_k,C_4)$ and $\ex(n,P_k,C_4)$.
	
	From now on we assume that $n = q^2 - 1$, where $q$ is a prime power. Let $\mathbb{F}$ be the field with $q$ elements. The vertex set of $G$ is $\mathbb{F}^2 \setminus \{(0,0)\}$ and a pair of vertices $(a,b),(c,d)$ are adjacent if and only if $ac + bd = 1$. Note that $(a,b) \in V(G)$ has a loop if and only if $a^2 + b^2 = 1$. The number of solutions to $x^2 + y^2 = 1$ is at most $2q$, since for every fixed $x \in \mathbb{F}$ there are at most $2$ solutions for $y$. This implies that the number of loops is at most $2q$.
	Note that for every
	$(a,b) \in V(G)$ there are $q$ solutions $(x,y)$ to
	$ax + by = 1$. Thus, the degree of every $(a,b) \in V(G)$ is either $q - 1$ or $q$, depending on whether or not $(a,b)$ has a loop. This implies that for every $k \geq 1$, $G$ contains at least $\frac{1}{2}n(q-1)(q-2)\dots(q-k) =
	\left( \frac{1}{2} - o(1) \right)n^{\frac{k}{2}+1}$ paths of length $k$.
	
	Observe that for every pair of vertices $(a,b),(c,d) \in V(G)$, there is at most one solution to the system $ax + by = cx + dy = 1$, implying that $(a,b)$ and $(c,d)$ have at most one common neighbour. This shows that $G$ is $C_4$-free. To finish the proof, it remains to show that the number of $k$-cycles in $G$ is as stated. Since this was proved for $k=3$ in \cite{Alon_Shikhelman}, we may assume from now on that $k \geq 5$.
	
	Note that if $(a,b),(c,d) \in V(G)$ are linearly independent and have no loops then they have a common neighbour. Indeed, by linear independence there is a (unique) solution to the system $ax + by = cx + dy = 1$. As $(a,b)$ and $(c,d)$ do not have loops, this solution is neither $(a,b)$ nor $(c,d)$, and hence it is a common neighbour of $(a,b)$ and $(c,d)$. As the number of loops in $G$ is at most $2q$, the number of pairs of vertices $(a,b),(c,d) \in V(G)$ for which either $(a,b)$ or $(c,d)$ has a loop is at most $2qn$. Furthermore, the number of collinear pairs
	$(a,b),(c,d) \in V(G)$ is $\frac{(q-1)n}{2}$. Therefore, all but
	$2qn + \frac{(q-1)n}{2} \leq 3nq$
	of the pairs of vertices are linearly independent and do not have loops, and hence have a common neighbour. We have thus proven the following.
	\begin{fact}\label{fact:common_neighbour_pairs}
		All but $3nq$ of the pairs of vertices in $G$ have a common neighbour.
	\end{fact}
	
	
	Note that for every $t \geq 2$ and $v_1,v_{t+1} \in V(G)$, the number of paths of length $t$ between $v_1$ and $v_{t+1}$ is at most $q^{t-2}$. Indeed, consider a path $v_1,\dots,v_{t+1}$. Since the maximal degree in $G$ is $q$, the number of choices of $v_2,\dots,v_{t-1}$ is at most $q^{t-2}$. Since $v_t$ is a common neighbour of $v_{t-1}$ and $v_{t+1}$, there is at most one choice for $v_t$ given $v_2,\dots,v_{t-1}$.
	
	A path is {\em good} if its endpoints have a common neighbour which is not on the path, and otherwise it is {\em bad}. To complete the proof, it is enough to show that for every $t \geq 3$, the number of bad paths of length $t$ is $O(nq^{t-1})$.
	Indeed, we already proved that $G$ contains at least
	$\left( \frac{1}{2} - o(1) \right)n^{\frac{k}{2}}$ paths of length $k-2$.
	Since the number of bad paths of length $k-2$ is
	$O(nq^{k-3}) = O \big( n^{\frac{k-1}{2}} \big)$, the number of good paths of length $k-2$ is at least $\left( \frac{1}{2} - o(1) \right)n^{\frac{k}{2}}$. A good path of length $k-2$ can be made into a $k$-cycle by adding the (unique) common neighbour of the endpoints of the path. Since every cycle contains $k$ subpaths of length $k-2$, the lemma follows.
	
	It thus remains to show that for every $t \geq 3$, the number of bad paths of length $t$ is $O(nq^{t-1})$. There are two types of bad paths: those whose endpoints do not have a common neighbour, and those whose endpoints have a common neighbour which is on the path. First, by Fact \ref{fact:common_neighbour_pairs}, the number of pairs of vertices $u,v \in V(G)$ which do not have a common neighbour is at most
	$3nq$.
	We proved that for each such $u,v$ there are at most $q^{t-2}$ paths of length $t$ between $u$ and $v$. Thus, there are at most $O(nq^{t-1})$ paths of length $t$ whose endpoints do not have a common neighbour. Second, let
	$u,v \in V(G)$ be vertices having a common neighbour and let $w$ be their unique common neighbour. The number of paths of length $t$ from $u$ to $v$ in which $w$ is at distance $i$ from $u$ (and hence at distance $t-i$ from $v$) is at most $q^{t-3}$ if $i \in \{1,t-1\}$ and at most $q^{i-2}q^{t-i-2} = q^{t-4}$ if $2 \leq i \leq t - 2$. By summing over $1 \leq i \leq t-1$ we get that the number of paths of length $t$ from $u$ to $v$ which contain $w$ is at most $2q^{t-3} + (t-3)q^{t-4} = O(q^{t-3})$. Since the number of choices for $u,v$ is at most $\binom{n}{2}$, the total number of paths of length $t$ that contain the common neighbour of their endpoints is $O(n^2q^{t-3}) = O(nq^{t-1})$. In conclusion, the number of bad paths is $O(nq^{t-1})$, as required.
\end{proof}

\noindent
We end this section by proving the lower bound in Proposition \ref{prop:C3_even_cycle}.
\begin{claim}\label{prop:C3_even_cycle_lower_bound}
	For every $\ell \geq 3$ we have
	$\ex(n,C_3,C_{2\ell}) =
	\Omega \big( \ex(n,\{C_4,C_6,\dots,C_{2\ell}\}) \big)$.
\end{claim}
\begin{proof}
	We use an argument similar to the one used in \cite{GL}.
	Let $G' = (A \cup B, E)$ be a maximum size
	$n \times n$ bipartite graph with no $C_4,C_6,\dots,C_{2\ell}$. Let $G$ be the graph obtained from $G'$ by replacing every vertex of $A$ by an edge (and replacing edges of $G'$ by copies of
	$K_{2,1}$). Then $G$ has $3n$ vertices, and
	one triangle per each edge of $G'$; so
	$G$ contains
	$e(G') \geq \frac{1}{2} \cdot \ex(2n,\{C_4,C_6,\dots,C_{2\ell}\})$ triangles. Now assume by contradiction that $C$ is a copy of $C_{2\ell}$ in $G$. By contracting the edges of $C$ inside $A$, we get a closed walk $C'$ in $G'$ of length at most $2\ell$. For each $a \in A$, let $a_1$ and $a_2$ denote the two ``copies" of $a$ in $G$. If for every $a \in C' \cap A$, only one of the copies of $a$ is in $C$, then $C' = C$, in contradiction to the $C_{2\ell}$-freeness of $G'$. So there is some $a \in A$ such that $a_1,a_2 \in C$. In the cycle $C$ there are two paths between $a_1$ and $a_2$, and since $|C| = 2\ell \geq 6$, one of these paths must have length at least $3$. Hence, there are distinct $b_1,b_2 \in B$ such that $(a_1,b_1),(a_2,b_2) \in E(G)$, and there is a path $P$ in $G$ between $b_1$ and $b_2$ which does not go through $a_1$ or $a_2$. Contracting $P$ gives a path $P'$ in $G'$ between $b_1$ and $b_2$, which does not go through $a$. Then $a,b_1,P,b_2,a$ is a cycle in $G'$ of length at most $2\ell$, in contradiction to the choice of $G'$.
\end{proof}

\end{document}